\documentclass[reqno,10pt]{amsart}

\usepackage{amsmath,amsfonts,amssymb,amsthm,epsfig,amscd}
\usepackage[]{fontenc}
\usepackage[latin1]{inputenc}
\usepackage{enumerate}
\usepackage[cmtip,all]{xy}
\usepackage{tabularx,multirow}
\usepackage[dvipsnames]{xcolor}
\usepackage[colorlinks=true]{hyperref}
\usepackage{url}
\usepackage{mathtools}
\usepackage{tikz}
\usepackage{tikz-cd}
\usepackage{enumitem}
\usepackage{float}

\newcommand{\bigslant}[2]{{\raisebox{.2em}{$#1$}\left/\raisebox{-.2em}{$#2$}\right.}}
\usepackage{dynkin-diagrams}
\def\row#1/#2!{#1_{\IfStrEq{#2}{}{n}{#2}} & \dynkin{#1}{#2}\\}

 \allowdisplaybreaks \numberwithin{equation}{section}

\newtheorem{thm}{Theorem}[section]
\newtheorem*{thm*}{Theorem}

\newtheorem{prop}[thm]{Proposition}
\newtheorem{cor}[thm]{Corollary}
\newtheorem{lem}[thm]{Lemma}

\newtheorem*{qst*}{Problem}

\theoremstyle{definition}
\newtheorem{defn}[thm]{Definition}

\theoremstyle{definition}
\newtheorem{rmk}[thm]{Remark}
\newtheorem{exmp}[thm]{Example}

\newcommand{\C}{\mathbb{C}}
\newcommand{\R}{\mathbb{R}}
\newcommand{\Z}{\mathbb{Z}}
\newcommand{\Ol}{\mathcal{O}}

\newcommand{\Pn}{\mathbb{P}}
\newcommand{\Q}{\mathbb{Q}}

\newcommand{\Pic}{\textnormal{Pic}}

\newcommand{\Nef}{\textnormal{Nef}}

\newcommand{\NS}{\textnormal{NS}}
\newcommand{\Aut}{\textnormal{Aut}}
\newcommand{\rk}{\textnormal{rk}}
\newcommand{\Mov}{\textnormal{Mov}}
\newcommand{\Bir}{\textnormal{Bir}}
\newcommand{\Deg}{\textnormal{deg}}

\newcommand{\Div}{\textnormal{Div}}
\newcommand{\Dim}{\textnormal{dim}}
\newcommand{\Sym}{\textnormal{Sym}}

\newlist{thmlist}{enumerate}{1}
\setlist[thmlist]{label=(\roman{thmlisti}), ref=\thethm.(\roman{thmlisti}),noitemsep}

\begin{document}
	\title[Geometric description of $\langle 2 \rangle$-polarised Hilbert squares of generic K3 surfaces]{Geometric description of $\langle 2 \rangle$-polarised Hilbert squares of generic K3 surfaces}
	
	\author{Simone Novario}
	\address{Universit\`a degli Studi di Milano, Dipartimento di Matematica ``F. Enriques'', Via Cesare Saldini 50, 20133 Milano, Italy. \newline \indent Universit\'e de Poitiers, Laboratoire de Math\'ematiques et Applications, T\'el\'eport 2, Boulevard Marie et Pierre Curie, 86962 Futuroscope Chasseneuil, France.}
	
	\email{simone.novario@unimi.it, simone.novario@math.univ-poitiers.fr}
	
	\begin{abstract}
		A generic K3 surface of degree $2t$ is a general complex projective K3 surface $S_{2t}$ whose Picard group is generated by the class of an ample divisor $H \in \text{Div}(S_{2t})$ such that $H^2=2t$ with respect to the intersection form. We show that if $X$ is the Hilbert square of a generic K3 surface of degree\,\,$2t$ with $t \neq 2$ which admits an ample divisor $D \in \text{Div}(X)$ with $q_X(D)=2$, where $q_X$ is the Beauville--Bogomolov--Fujiki form, then $X$ is a double EPW sextic.
	\end{abstract}

\maketitle

\section{Introduction}
A classical result in the theory of complex projective K3 surfaces says that if a K3 surface admits an ample divisor\,\,$D$ with $D^2=2$ with respect to the intersection form, then the complete linear system $|D|$ is basepoint free and the morphism that it induces is a double cover of the plane $\Pn^2$ ramified on a sextic curve. This was proved by Saint--Donat in \cite{saint1974projective}.

It is natural to study a similar problem for a projective IHS (Irreducible Holomorphic Symplectic) manifold of dimension $2n$ with $n \ge 2$, a sort of higher dimensional generalization of K3 surfaces. The importance of these varieties is linked to the Beauville--Bogomolov decomposition theorem, see \cite{beauville1983varietes}: up to a finite étale cover, any compact Kähler manifold with trivial first Chern class is the product of a complex torus, irreducible Calabi--Yau manifolds and IHS manifolds.

Let $S$ be a K3 surface. We denote by $T(S)$ its transcendental lattice, which coincides with $\NS(S)^{\perp}$, where $\NS(S)$ is the Néron--Severi group of $S$ and the orthogonal is taken with respect to the intersection form. Then $T(S)_{\Q}:=T(S) \otimes \Q$ is a rational Hodge structure of weight $2$ and we denote by
\begin{equation*}
E_S:=\text{Hom}_0(T(S)_{\Q}, T(S)_{\Q})
\end{equation*}
the algebra of endomorphisms of weight $0$ on $T(S)_{\Q}$. See \cite[$\S 3$]{huybrechts2016lectures} for details on Hodge structures. By a result of Zahrin, see \cite{Zarhin1983}, $E_S$ is either $\Q$, or a totally real field or a CM field. We say that $S$ is a K3 surface \emph{general in its rank}, or simply \emph{general}, if $E_S \cong \Q$. If $E_S$ is a totally real field then $S$ has \emph{real multiplication} (RM), while if\,\,$E_S$ is a CM field then $S$ has \emph{complex multiplication} (CM). Then K3 surfaces with RM or CM describe a locus of positive codimension in the analytic moduli space of K3 surfaces of fixed Picard rank, except for K3 surfaces of Picard rank\,\,$20$, which have all CM. See \cite{van2008real} and \cite{elsenhans2016point} for details. 

A \emph{generic K3 surface} is a general projective K3 surface whose Picard group has rank\,\,$1$. We say that a generic K3 surface $S_{2t}$ has \emph{degree} $2t$ if $\Pic(S_{2t}) = \Z H$ with $H^2=2t$. In this paper we focus on Hilbert schemes of $2$ points on a generic K3 surface, also known as \emph{Hilbert squares} of generic K3 surfaces: the Hilbert square of a K3 surface $S$ will be denoted by $S^{[2]}$. The main problem of this paper is to determine the geometric description of a $\langle 2 \rangle$-polarised Hilbert square $X$ of a generic K3 surface, i.e., $X$ admits an ample divisor $D \in \Div(X)$ with $q_X(D)=2$, where $q_X$ is the \emph{Beauville--Bogomolov--Fujiki} quadratic form on $H^2(X, \Z)$. This can be seen as a generalization of the problem studied by Saint--Donat presented in the first lines. We denote by $|D|$ the complete linear system associated to $D \in \Div(X)$, i.e., the set of effective divisors linearly equivalent to $D$, and by $\text{Bs}|D|$ the base locus of $|D|$, which is the set of points on\,\,$X$ where all the global sections in $H^0(X, \Ol_X(D))$ vanish. Our aim is to study the rational map $\varphi_{|D|}$ induced by the complete linear system $|D|$, i.e.,
\begin{equation*}
	\varphi_{|D|}: X \dashrightarrow \Pn(H^0(X,\ \Ol_X(D))^{\vee}),
\end{equation*}
where by definition $x \in X \setminus \text{Bs}|D|$ is mapped to the hyperplane in $H^0(X, \Ol_X(D))$ consisting of sections vanishing at\,\,$x$. If $\text{Bs}|D|$ is empty, we say that $|D|$ is \emph{basepoint free}: in this case $\varphi_{|D|}$ is a morphism.

Let $X$ and $D$ be as above. Then there exists an anti-symplectic involution $\iota$ which generates the group $\Aut(X)$ of biregular automorphisms on $X$ by a result obtained by Boissière, Cattaneo, Nieper-Wi{\ss}kirchen and Sarti in \cite{boissiere2016automorphism}, and $\iota$ is such that $\iota^*[D]=[D]$ in the Néron--Severi group $\NS(X)$. Here anti-symplectic means that $\iota^*\sigma_X=-\sigma_X$, where $\sigma_X \in H^0(X, \Omega^2_X)$ is a symplectic form of $X$. The main theorem of this paper, which we now present, gives a geometrical description of the map induced by the complete linear system $|D|$, cf.\,Theorem \ref{thm class F and comm diagram} and Theorem \ref{thm main}.
\begin{thm*}
	Let $X$ be the Hilbert square of a generic K3 surface $S_{2t}$ of degree $2t$ such that $X$ admits an ample divisor $D \in \Div(X)$ with $q_X(D)=2$. Suppose that $t \neq 2$, and denote by $\iota$ the anti-symplectic involution which generates the group $\Aut(X)$. Then the complete linear system $|D|$ is basepoint free, the morphism
	\begin{equation*}
		\varphi_{|D|}: X \rightarrow Y \subset \Pn^5
	\end{equation*}
	is a double cover whose ramification locus is the surface $F$ of points fixed by $\iota$, and $Y \cong X/\langle \iota \rangle$ is an EPW sextic, in particular $X$ is a double EPW sextic. Moreover, if $H^{2,2}(X, \Z):=H^4(X, \Z) \cap H^{2,2}(X)$, then
	\begin{equation*}
	[F]=5D^2-\frac{1}{3}c_2(X) \in H^{2,2}(X, \Z),
	\end{equation*}
	where $c_2(X) \in H^4(X, \Z)$ is the second Chern class of $X$.
\end{thm*}
Here a double EPW sextic is a double cover of an EPW sextic ramified in its singular locus: see \cite{eisenbud2001lagrangian} for the definition of EPW sextic and \cite{o2006irreducible} for details on double EPW sextics. The paper \cite{o2008irreducible}, where O'Grady gives a classification, up to deformation equivalence, of \emph{numerical} $K3^{[2]}$, will play an important role. A numerical $K3^{[2]}$ is by definition an IHS manifold $M$ which admits an isomorphism of abelian groups $\psi: H^2(M, \Z) \rightarrow H^2(S^{[2]}, \Z)$ for some K3 surface $S$ such that $\int_M \alpha^4=\int_{S^{[2]}}\psi(\alpha)^4$ for every $\alpha \in H^2(M, \Z)$. In particular he showed that a numerical\,\,$K3^{[2]}$ is deformation equivalent either to a double EPW sextic or to an IHS manifold $Z$ admitting a rational map $f: Z \dashrightarrow \Pn^5$ which is birational onto its image $Y$, with $6 \le \Deg(Y) \le 12$. The link between our problem and the one studied in \cite{o2008irreducible} is given by the fact that O'Grady proved the following: a numerical $K3^{[2]}$ is deformation equivalent to an IHS manifold $Z$ of $K3^{[2]}$-type such that $\Pic(Z)$ is generated by the class of an ample divisor $H \in \Pic(Z)$ with $q_Z(H)=2$. Let $X$ be the Hilbert square of a general K3 surface of degree $2t$ and suppose that there is an ample divisor $D \in \Div(X)$ with $q_X(D)=2$. In this specific case the result obtained by O'Grady implies that $X$ is deformation equivalent to a double EPW sextic, while our main theorem is stronger, since it shows that $X=S^{[2]}_{2t}$, if $t \neq 2$, is exactly a double EPW sextic, giving a more precise geometric description of these varieties. The case $t=2$ was already studied by Welters and Beauville in \cite{welters1981abel} and \cite{beauville1983some}: the map\,\,$\varphi_{|D|}$ is a finite morphism of degree $6$ with image $\mathbb{G}(1, \Pn^3)$, the Grassmannian of lines in\,\,$\Pn^3$.

Our strategy is to follow \cite{o2008irreducible}, using the anti-symplectic involution which generates $\Aut(X)$ given by \cite{boissiere2016automorphism} in order to get as much information as possible on the geometry of the complete linear system.

The paper is organised as follows. In Section\,\,\ref{Section Pell} we give basics on Pell equations and Pell-type equations. In Section\,\,\ref{section generalities IHS manifolds} we recall the definition of IHS manifold, together with main properties of this family of varieties. In Section\,\,\ref{section generalities IHS K32 type} we present some useful results on IHS manifolds of $K3^{[2]}$-type, in particular we recall the description of the lattice $H^{2,2}(S^{[2]}, \Z):=H^4(S^{[2]}, \Z) \cap H^{2,2}(S^{[2]})$ of integral Hodge classes of type $(2, 2)$ on $S^{[2]}$ obtained in \cite{novario2021hodge} for a general projective K3 surface\,\,$S$: this will be extremely useful in several steps of this paper. In Section \ref{section generalities HS general K3 surfaces} we recall the most important results on the Hilbert square $X$ of a generic K3 surface, in particular the description of the group of biregular automorphisms obtained in \cite{boissiere2016automorphism}: $\Aut(X)$ is not trivial and generated by a non-natural anti-symplectic involution $\iota$ when\,\,$X$ is the Hilbert square of a generic K3 surface of degree $2t$ admitting an ample divisor $D \in \Div(X)$ with $q_X(D)=2$. We also briefly describe the case $t=2$, already studied by Welters and Beauville in \cite{welters1981abel} and \cite{beauville1983some}. From now on $X$ is the Hilbert square of a generic K3 surface admitting an ample divisor\,\,$D$ with $q_X(D)=2$. In Section\,\,\ref{Section fixed locus} we study the surface $F$ of points on $X$ fixed by the anti-symplectic involution $\iota$: we compute the class of\,\,$F$ in $H^{2,2}(X, \Z)$ and we show that the rational map\,\,$\varphi_{|D|}$ factors through the quotient with respect to the involution\,\,$\iota$. In Section\,\,\ref{Section irreducibility property} we prove the \emph{irreducibility property}: if $D_1, D_2 \in |D|$ are distinct and $t \neq 2$, where $2t$ is the degree of the underlying generic K3 surface, then $D_1 \cap D_2$ is a reduced and irreducible surface. This is a fundamental step in the geometrical analysis of the rational map $\varphi_{|D|}$. When $t=2$, the surface $D_1 \cap D_2$ can be reducible: we give a geometrical interpretation of the two irreducible components. We conclude with Section\,\,\ref{Section geometric description}, where we show the main result of this paper: keeping notation as above, if $t \neq 2$ then\,\,$X$ is a double EPW sextic and $\varphi_{|D|}$ is the double cover associated.

The topic of this paper is related to \cite[$\S 7$]{debarre2019period}. The approach and the techniques that we use are different: main tools are results on integral Hodge classes of type $(2, 2)$ on the Hilbert square of a generic projective K3 surface obtained in \cite{novario2021hodge}, which are exploited in particular in Theorem \ref{thm class F and comm diagram} and in Theorem \ref{thm irreducibility property}, and the papers \cite{o2008irreducible} and \cite{boissiere2016automorphism}. Moreover, the irreducibility property of Theorem \ref{thm irreducibility property} explains geometrically why the case $t \neq 2$ is special, compare with \cite[Remark 7.7]{debarre2019period}.

This paper is based on Chapter 4 and Chapter 5 of the author's PhD thesis, see \cite{novario2021ths}.

\textbf{Acknowledgements}: I thank my advisors Samuel Boissière and Bert van Geemen for the constant support I received. I would like to thank Michela Artebani and Alice Garbagnati for useful discussions on Section \ref{Section fixed locus}. I also thank Pietro Beri and Ángel David Ríos Ortiz for many useful discussions.
\section{Pell equations and Pell-type equations} \label{Section Pell}
We give a brief overview of Pell equations and Pell-type equations.
\begin{defn}
A \emph{Pell-type equation} is a diophantine equation of the form
	\begin{equation*}
		x^2-dy^2=n,
	\end{equation*}
	where $d \in \Z_{>0}$ is a positive integer, $n \in \Z\setminus\{0\}$ is a non-zero integer and $x,y$ are variables. We denote an equation of this form by $P_d(n)$. We call $P_d(1)$ a \emph{Pell equation} and $P_d(-1)$ a \emph{negative Pell equation}.
\end{defn}
We are interested in integral solutions of Pell equations and Pell-type equations. If $d=c^2$ is a square, the only solutions of the Pell equation $P_d(1)$ are $(x, y)=(\pm 1, 0)$, and the Pell-type equation $P_d(n)$ can be written as $(x+cy)(x-cy)=n$, so it can be easily solved. From now on, we will assume that $d$ is not a square. In this case the Pell-type equation $P_d(n)$ can be written in $\Z[\sqrt{d}\,]:=\bigslant{\Z[x]}{(x^2-d)}$ as
\begin{equation*}
	(x+y\sqrt{d})(x-y\sqrt{d})=n.
\end{equation*}
Let $z=x+y\sqrt{d} \in \Z[\sqrt{d}]$: we define its \emph{conjugate} as $\bar{z}:=x-y\sqrt{d} \in \Z[\sqrt{d}]$, and its \emph{norm} as $N(z):=z\overline{z}=x^2-dy^2 \in \Z$. With this notation, $P_d(1)$ and $P_d(n)$ can be written respectively as $N(z)=1$ and $N(z)=n$, where $z=x+y\sqrt{d}$.
\begin{defn} \label{defn equiv sol Pell}
	Given a Pell equation $P_d(1)$ or a Pell-type equation $P_d(n)$, two solutions $(X, Y)$ and $(X^{\prime}, Y^{\prime})$ are said to be \emph{equivalent} if
	\begin{equation*}
		\frac{XX^{\prime}-dYY^{\prime}}{n} \in \Z, \qquad \frac{XY^{\prime}-X^{\prime}Y}{n} \in \Z.
	\end{equation*}
\end{defn}
Note that two solutions $(X, Y)$ and $(X^{\prime}, Y^{\prime})$ of a Pell-type equation $P_d(n)$ are equivalent if $\frac{z_1 \cdot \overline{z_2}}{n} \in \Z[\sqrt{d}]$, where $z_1=X+Y\sqrt{d}$ and $z_2=X^{\prime}+Y^{\prime}\sqrt{d}$.
\begin{defn}
	Consider a Pell equation or a Pell-type equation. The \emph{fundamental} solution $(X, Y)$ in an equivalence class of solutions is the one with smallest non-negative $Y$ if such a solution is unique in its class. Otherwise, there are two solutions $(X, Y)$, $(-X, Y)$, which are said to be \emph{conjugated}, with smallest non-negative $Y$: the fundamental solution is the one of the form $(X, Y)$ with $X>0$.
\end{defn}
One can show that a Pell equation $P_d(1)$ is always solvable and all its solutions are equivalent. If $z_0=a+b\sqrt{d}$ is the fundamental solution, then all the other solutions are of the form
\begin{equation*}
	z=\pm z_0^m, \qquad m \in \Z_{>0}.
\end{equation*}
Similarly, if $z_0$ is a fundamental solution of $P_d(n)$ and $\tilde{z}_0$ is the fundamental solution of $P_d(1)$, all the other solutions equivalent to $z_0$ are of the form
\begin{equation*}
	z=\pm \tilde{z}_0^m \cdot z_0, \qquad m \in \Z_{>0}.
\end{equation*}
We now recall the definition of \emph{positive} and \emph{minimal} solution of a Pell-type equation.
\begin{defn}
	Let $P_d(n)$ be a Pell-type equation. A solution $(X, Y)$ is \emph{positive} if $X, Y > 0$. The \emph{minimal} solution of $P_d(n)$ is the positive solution with smallest $X$.
\end{defn}
\begin{rmk} \label{rmk Pell (a, b) and (c, d)}
	The minimal solution of the Pell equation $P_t(1)$ coincides with the square of the minimal solution of the negative Pell equation $P_t(-1)$, if this exists, i.e., if $a+b\sqrt{t} \in \Z[\sqrt{t}]$ is the minimal solution of $N(z)=-1$ and $c+d\sqrt{t} \in \Z[\sqrt{t}]$ is the minimal solution of $N(z)=1$, then $c+d\sqrt{t}=(a+b\sqrt{t})^2=a^2+b^2+2ab\sqrt{t}$, hence
	\begin{equation} \label{eq relation Pt(1) and Pt(-1)}
		a^2+tb^2=c, \qquad d=2ab.
	\end{equation}
\end{rmk}
We conclude this section with the following elementary result, which will be useful in Section \ref{Section fixed locus} and in Section \ref{Section irreducibility property}.
\begin{prop} \label{prop D=bh-adelta then b odd}
	Let $(a, b)$ be the minimal solution of the Pell-type equation $P_d(-1)$. Then $b$ is odd.
\end{prop}
\begin{proof}
	Suppose that $b$ is even. Then $a^2+1$ is divisible by $4$, since $a^2-db^2=-1$, i.e., $a^2+1=4X$ for some $X \in \Z$.
	\begin{itemize}
		\item If $a$ is even, i.e., $a=2Y$ for some $Y \in \Z$, then $4Y^2+1=4X$, which is not possible.
		\item If $a$ is odd, i.e., $a=2Y+1$ for some $Y \in \Z$, then $4Y^2+1+4Y+1=4X$, which gives $4(X-Y^2-Y)=2$, which is not possible.
	\end{itemize}
	We conclude that $b$ is odd.
\end{proof}
\section{Generalities on IHS manifolds} \label{section generalities IHS manifolds}
In this section we recall basics on irreducible holomorphic symplectic manifolds.
\begin{defn}
	An \emph{irreducible holomorphic symplectic} (IHS) \emph{manifold} is a simply connected compact complex Kähler manifold $X$ such that $H^0(X, \Omega^2_X)$ is generated by a non-degenerate holomorphic $2$-form, called \emph{symplectic form}.
\end{defn}
The definition of IHS manifold generalises to higher dimensions the one of K3 surface, the only example of IHS manifold of dimension $2$, as shown by the Enriques--Kodaira classification of compact complex surfaces. The existence of a symplectic form implies that the dimension of an IHS manifold is necessarily even. If $X$ is an IHS manifold, then the $\C$-vector space $H^0(X, \Omega^p_X)$ is zero if $p$ is odd, and it is generated by $\sigma^{\frac{p}{2}}$ if $0 \le p \le \Dim(X)$ is even, where\,\,$\sigma$ is a symplectic form, see \cite[Proposition 3]{beauville1983varietes}. The Picard group $\Pic(X)$ and the Néron--Severi group, defined as $\NS(X):=H^{1,1}(X)_{\R} \cap H^2(X, \Z)$, are isomorphic: $\Pic(X)$ then embeds in the second cohomology group $H^2(X, \Z)$. We recall that a \emph{lattice} is a free $\Z$-module $L$ of finite rank with a symmetric bilinear form $b: L \times L \rightarrow \Z$: we denote by $q: L \rightarrow \Z$ the quadratic form $q(x):=b(x, x)$ for every $x \in L$. If $\mathcal{B}:=\{e_1, \dots , e_n\}$ is a $\Z$-basis of $L$, the \emph{Gram matrix} of $L$ associated to $\mathcal{B}$ is the $n \times n$ symmetric matrix
\begin{equation*}
	\begin{pmatrix}
		b(e_1, e_1) & \cdots & b(e_1, e_n) \\
		\vdots & \ddots & \vdots \\
		b(e_n, e_1) & \cdots & b(e_n, e_n)
	\end{pmatrix}.
\end{equation*}
We say that a lattice $L$ of rank $n$ is \emph{non-degenerate} if for any non-zero $l \in L$ there exists $l^{\prime} \in L$ such that $b(l, l^{\prime}) \neq 0$, equivalently, $\text{det}(G) \neq 0$ if $G$ is a Gram matrix of $L$. A lattice $L$ is \emph{even} if $b(l, l) \in 2\Z$ for every $l \in L$, and \emph{odd} if it is not even. The \emph{determinant} of a lattice $L$ is the determinant of a Gram matrix $G$ of the lattice, and the \emph{discriminant} of $L$ is $\text{disc}(L):=|\text{det}(G)|$. A lattice $L$ is \emph{unimodular} if $\text{disc}(L)=1$. If $L$ is a unimodular lattice, for every $x \in L$ there exists $y \in L$ such that $b(x, y)=1$. A \emph{sublattice} of a lattice $L$ is a free submodule $L^{\prime} \subseteq L$ with symmetric bilinear form $b^{\prime}:=b|_{L^{\prime} \times L^{\prime}}$. A sublattice $L^{\prime} \subseteq L$ is \emph{primitive} if $L/L^{\prime}$ is a free module. We define the \emph{direct sum} of two lattices $L_1$ and $L_2$ as the lattice $L_1 \oplus L_2$ whose bilinear form is $b(v_1+v_2, w_1+w_2):=b_1(v_1, w_1)+b_2(v_2, w_2)$ for every $v_1, w_1 \in L_1$ and $v_2, w_2 \in L_2$, where $b_1$ and $b_2$ are the bilinear forms of $L_1$ and $L_2$ respectively. If $L$ and $L^{\prime}$ are two lattices with bilinear forms $b$ and $b^{\prime}$ respectively, we call \emph{morphism of lattices} $\varphi: L \rightarrow L^{\prime}$ a morphism of $\Z$-modules such that for every $l_1, l_2 \in L$ we have $b(l_1, l_2)=b^{\prime}(\varphi(l_1), \varphi(l_2))$. Note that morphisms between two non-degenerate lattices are injective. We say that a lattice \emph{embeds primitively} in a lattice $L^{\prime}$ if there is a morphism $\varphi: L \rightarrow L^{\prime}$ such that $\varphi(L)$ is a primitive sublattice of $L^{\prime}$. An \emph{isometry} is a bijective morphism of lattices. The \emph{divisibility} of an element $l \in L$ in a lattice $L$ is the positive generator of the ideal $\{b(l, m) \, | \, m \in L\} \subseteq \Z$.

For a lattice $L$ of rank $n$ we write $L_{\R}:=L \otimes_{\Z} \R$ and we extend $\R$-bilinearly the bilinear form $b$ to $L_{\R}$, similarly we extend $q$ to $L_{\R}$. If the lattice is non-degenerate, the \emph{signature} of $L$ is the signature $(l_{(+)}, l_{(-)})$ of the quadratic form on $L_{\R}$. A non-degenerate lattice is \emph{positive definite} if $l_{(-)}=0$, similarly it is \emph{negative definite} if $l_{(+)}=0$, while it is \emph{indefinite} if $l_{(+)}, l_{(-)} \neq 0$. 
\begin{exmp} \label{exmp lattice <k>}
	If $k$ is a non-zero integer, let $\langle k \rangle$ be the rank one lattice $L=\Z e$ with bilinear form $b(e, e)=k$.
\end{exmp}
\begin{exmp} \label{exmp lattice U}
	Let $U$ be the \emph{hyperbolic lattice}, i.e., the unique unimodular lattice of rank $2$ and signature $(1, 1)$. Its Gram matrix is the following:
	\begin{equation*} 
		\begin{pmatrix}
			0 & 1 \\
			1 & 0
		\end{pmatrix}.
	\end{equation*}
\end{exmp}
\begin{exmp} \label{exmp E_8(-1)}
	Let $E_8(-1)$ be the even unimodular lattice of signature $(0, 8)$ whose Gram matrix is the following:
	\begin{equation*} 
		\begin{pmatrix}
			-2 & 1 &  &  &  &  &  &  \\
			1 & -2 & 1 &  &  &  &  &  \\
			& 1 & -2 & 1 &  &  &  & 1 \\
			&  & 1 & -2 & 1 &  &  &  \\
			&  &  & 1 & -2 & 1 &  &  \\
			&  &  &  & 1 & -2 & 1 &  \\
			&  &  &  &  & 1 & -2 &  \\
			&  & 1 &  &  &  &  & -2 \\
		\end{pmatrix}.
	\end{equation*}
\end{exmp}
Let $X$ be an IHS manifold. By the universal coefficient theorem, the second singular cohomology group $H^2(X, \Z)$ is torsion free, and this can be equipped with a non-degenerate integral quadratic form by the following result due to Beauville, Bogomolov and Fujiki, see \cite{beauville1983varietes} and \cite{fujiki1987rham}.
\begin{thm}[Beauville--Bogomolov--Fujiki form] \label{thm BBF form}
	Let $X$ be an IHS manifold of dimension $2n$. Then there exists an integral quadratic form $q_X: H^2(X, \Z) \rightarrow \Z$ and a constant $c_X \in \Q_{>0}$ such that
	\begin{equation*}
		\displaystyle\int_X \alpha^{2n} = c_X \frac{(2n)!}{n!2^n}q_X(\alpha)^n \qquad \text{for all}\,\,\alpha \in H^2(X, \Z).
	\end{equation*}
	The quadratic form $q_X$ is called Beauville--Bogomolov--Fujiki (BBF) form, and $c_X$ is called Fujiki constant of $X$. Moreover, $(H^2(X, \Z), q_X)$ is a lattice of signature $(3, b_2(X)-3)$, where $b_2(X)$ is the second Betti number of $X$.
\end{thm}
An example of IHS manifold of dimension $2n$, where $n \ge 2$, is given by the \emph{Hilbert scheme of} $n$ \emph{points on a K3 surface} $S$, the scheme which parametrises zero-dimensional closed subschemes of length\,\,$n$ on a K3 surface: we denote it by $S^{[n]}$. Let $S^{(n)}$ be the quotient of $S^n=S \times \dots \times S$ by the symmetric group of $n$ elements, so $S^{(n)}$ is the variety of $0$-cycles of degree $n$. Then the \emph{Hilbert--Chow morphism} $\rho: S^{[n]} \rightarrow S^{(n)}$ is defined as follows: a point $[\xi] \in S^{[n]}$ is mapped to the cycle $\sum_x l(\Ol_{\xi, x}) \cdot x$, see for instance \cite{iversen2006linear}. The singular locus of $S^{(n)}$ is the so-called diagonal, i.e., the set of cycles $p_1+\dots + p_n$ such that there exist distinct $i$ and $j$ with $p_i=p_j$. The Hilbert--Chow morphism is a desingularization of $S^{(n)}$, and the pre-image of the diagonal is an irreducible divisor $E$ on $S^{[n]}$. The Hilbert scheme of $n$ points on a K3 surface is an IHS manifold by \cite[Théorème 3]{beauville1983varietes}. There exists a primitive class $\delta \in \Pic(S^{[n]})$ such that $2\delta=[E]$. Moreover, there is a primitive embedding of lattices
\begin{equation*}
	i: H^2(S, \Z) \hookrightarrow H^2(S^{[n]}, \Z)
\end{equation*}
such that $H^2(S^{[n]}, \Z)=i(H^2(S, \Z)) \oplus \Z\delta$, and $q_{S^{[n]}}(\delta)=-2(n-1)$. Recall that $H^2(S, \Z) \cong U^{\oplus 3} \oplus E_8(-1)^{\oplus 2}$, in particular $H^2(S, \Z)$ is an even unimodular lattice of signature $(3, 19)$, see for instance \cite[$\S \text{VII.3}$]{barth2015compact}. Then there is an isometry of lattices
\begin{equation*}
	H^2(S^{[n]}, \Z) \cong U^{\oplus 3} \oplus E_8(-1)^{\oplus 2} \oplus \langle -2(n-1) \rangle,
\end{equation*}
similarly $\Pic(S^{[n]}) =i(\Pic(S)) \oplus \Z \delta$: see \cite[$\S 6$]{beauville1983varietes} for details. The Fujiki constant of the Hilbert scheme of $n$ points on a K3 surface $S$ is $c_{S^{[n]}}=1$, see \cite[$\S 9$]{beauville1983varietes}. In particular for K3 surfaces the BBF form coincides with the intersection form. Moreover, the singular cohomology ring $H^*(S^{[n]}, \Z)$ for a K3 surface $S$ and $n \ge 1$ is torsion free by \cite[Theorem 1]{markman2007integral}. When $n=2$, the variety $S^{[2]}$ is usually called \emph{Hilbert square of a K3 surface}. An IHS manifold which is deformation equivalent to the Hilbert square of a K3 surface is said to be of $K3^{[2]}$-\emph{type}.

The other known examples of IHS manifolds up to deformation equivalence are \emph{generalised Kummer varieties}, see \cite[$\S 7$]{beauville1983varietes}, an isolated example of dimension $10$ and second Betti number $b_2=24$, see \cite{o1999desingularized}, and an isolated example of dimension $6$ and second Betti number $b_2=8$, see \cite{o2003new}. We do not discuss details on these examples since in this paper we deal only with Hilbert squares of K3 surfaces.

We conclude this section with the following useful correspondence between primitive elements in $H^2(X, \Z)$ and primitive elements in $H_2(X, \Z)_f$, where $X$ is an IHS manifold and $H_2(X, \Z)_f$ is the torsion free quotient group of the second singular homology $H_2(X, \Z)$. See \cite{hassett2001rational} for details.
\begin{prop}[Hassett--Tschinkel] \label{Prop Hassett Tschinkel}
	Let $X$ be an IHS manifold and denote by $H_2(X, \Z)_f$ the torsion free quotient group of $H_2(X, \Z)$. Let $(\, \cdot \, , \, \cdot \,)$ be the BBF bilinear form. Then there is a correspondence between primitive elements in $H^2(X, \Z)$ and primitive elements in $H_2(X, \Z)_f$. In particular:
	\begin{enumerate}[label=(\roman*)]
		\item For every primitive element $R \in H_2(X, \Z)_f$ there exists a unique class $\omega \in H^2(X, \Q)$ such that
		\begin{equation*}
			\epsilon(R \cap v)=(\omega, v) \qquad \text{for every}\,\,v \in H^2(X, \Z),
		\end{equation*}
		where $\epsilon: H_0(X, \Z) \xrightarrow{\sim} \Z$ is the isomorphism in \cite[Theorem $\textnormal{IV}.2.1$]{bredon2013topology} and $\cap$ is the cap product. The primitive $\rho \in H^2(X, \Z)$ associated to $R$ is the primitive element such that $c\rho=\omega$ for some $c \in \Q_{>0}$.
		\item For every primitive element $\rho \in H^2(X, \Z)$ of divisibility $\textnormal{div}(\rho)=d$ in $(H^2(X, \Z), q_X)$, there exists a unique primitive $R \in H_2(X, \Z)_f$ such that 
		\begin{equation*}
			d\cdot \epsilon (R \cap v)=(\rho, v) \qquad \text{for every}\,\,v \in H^2(X, \Z).
		\end{equation*}
	\end{enumerate}
\end{prop}
\section{Generalities on IHS manifolds of $K3^{[2]}$-type} \label{section generalities IHS K32 type}
Let $X$ be an IHS manifold of dimension $4$ of $K3^{[2]}$-type. In this section we recall the link between the intersection pairing on $H^4(X, \Q)$ and the $\Q$-bilinear extension on $H^2(X, \Q)$ of the BBF form, and we introduce the dual $q_X^{\vee}$ of the BBF quadratic form. We refer to \cite[$\S 2$]{o2008irreducible}.

First of all, we state the following corollary of Verbitsky's results in \cite{verbitsky1996cohomology}, obtained by Guan in \cite{guan2001betti}, see also \cite[Corollary 2.5]{o2010higher}. We denote by $b_i(X)$ the $i$-th Betti number of $X$.

\begin{prop} \label{prop Guan}
	Let $X$ be an IHS manifold of dimension $4$. Then $b_2(X) \le 23$. If equality holds then $b_3(X)=0$ and the map
	\begin{equation*}
		\textnormal{Sym}^2H^2(X, \Q) \rightarrow H^4(X, \Q)
	\end{equation*}
	induced by the cup product is an isomorphism. In particular this happens when $X$ is an IHS fourfold of $K3^{[2]}$-type.
\end{prop}
Since $X$ is a compact complex manifold of dimension $\Dim_{\C}(X)=4$, the singular cohomology group $H^4(X, \Z)$ has an intersection pairing induced by the cup product:
\begin{equation*}
	\langle \, \cdot \, , \, \cdot \, \rangle : H^4(X, \Z) \times H^4(X, \Z) \rightarrow \Z, \qquad \langle \alpha, \beta \rangle :=\displaystyle\int_X \alpha \beta.
\end{equation*}
We write $\langle \cdot \, , \, \cdot \, \rangle$ also for the $\Q$-bilinear extension of the intersection pairing above to $H^4(X, \Q) \times H^4(X, \Q)$, obtaining a $\Q$-valued intersection pairing on $\Sym^2H^2(X, \Q)$. Let $X$ be an IHS manifold of $K3^{[2]}$-type: the following relation between $\langle \, \cdot \, , \, \cdot \, \rangle$ and the $\Q$-extension of the BBF form on $H^2(X, \Q)$ holds, see \cite[Remark 2.1]{o2008irreducible}.
\begin{prop}[O'Grady] \label{prop rmk 2.1 OG}
	Let $X$ be an IHS fourfold of $K3^{[2]}$-type. The intersection pairing $\langle \, \cdot \, , \, \cdot \, \rangle$ defined above is the bilinear form on $\textnormal{Sym}^2H^2(X, \Q)$ given by
	\begin{equation*}
		\langle \alpha_1\alpha_2, \alpha_3\alpha_4 \rangle =(\alpha_1, \alpha_2)(\alpha_3, \alpha_4)+(\alpha_1, \alpha_3)(\alpha_2, \alpha_4)+(\alpha_1, \alpha_4)(\alpha_2, \alpha_3)
	\end{equation*}
	for every $\alpha_1, \alpha_2, \alpha_3, \alpha_4 \in H^2(X, \Q)$, where $(\,\cdot \, , \, \cdot \, )$ denotes the BBF bilinear form on $H^2(X, \Q)$.
\end{prop}
Let $q_X$ be the BBF quadratic form on $X$. Let $\{e_1, \dots , e_{23}\}$ be a basis of $H^2(X, \Q)$ and $\{e_1^{\vee}, \dots , e_{23}^{\vee}\}$ be the dual basis in $H^2(X, \Q)^{\vee}$, i.e., $e_i^{\vee}(e_j)=\delta_{i,j}$. Then we have 
\begin{equation*}
	q_X=\displaystyle\sum_{i,j}g_{i,j}e_i^{\vee}\otimes e_j^{\vee}, \qquad q_X^{\vee}=\displaystyle\sum_{i,j}m_{i,j}e_ie_j,
\end{equation*}
where $g_{i,j}:=(e_i, e_j)$, the matrix $(g_{i,j})$ is symmetric and $(m_{i,j})=(g_{i,j})^{-1}$. The values of the products $\langle q_X^{\vee}, \alpha \rangle$ for every $\alpha \in H^4(X, \Q)$ are given by the following proposition, see \cite[Proposition 2.2]{o2008irreducible}.
\begin{prop}[O'Grady] \label{prop O'Grady int q_Xvee}
	Let $X$ be an IHS fourfold of $K3^{[2]}$-type. Let $\langle \, \cdot \, , \, \cdot \, \rangle$ be the bilinear form described in Proposition $\ref{prop rmk 2.1 OG}$. Then $\langle \, \cdot \, , \, \cdot \, \rangle$ is non-degenerate and 
	\begin{equation*}
		\begin{array}{l}
			\langle q_X^{\vee}, \alpha \beta \rangle = 25(\alpha, \beta) \qquad \text{for all}\,\,\alpha, \beta \in H^2(X, \Q), \\[1ex]
			\langle q_X^{\vee}, q_X^{\vee} \rangle =23 \cdot 25.
		\end{array}
	\end{equation*} 
\end{prop}
Recall that \emph{rational Hodge classes} and \emph{integral Hodge classes} of type $(k, k)$ on a projective manifold $Y$ are elements belonging respectively to
\begin{equation*}
	H^{k,k}(Y, \Q):= H^{2k}(Y, \Q) \cap H^{k,k}(Y), \qquad H^{k,k}(Y, \Z):= H^{2k}(Y, \Z) \cap H^{k,k}(Y).
\end{equation*}
O'Grady has shown in \cite[$\S 3$]{o2008irreducible} that $q_X^{\vee}$ is a rational multiple of $c_2(X)$, the second Chern class of the tangent bundle of $X$, in particular it is an element of $H^{2,2}(X, \Q)$, i.e., it is a rational Hodge class of type $(2, 2)$. 
\begin{prop}[O'Grady] \label{prop O'Grady c_2(X)}
	Let $X$ be an IHS fourfold of $K3^{[2]}$-type. Then $q_X^{\vee} \in H^{2,2}(X, \Q)$, i.e., $q_X^{\vee}$ is a rational Hodge class of $X$ of type $(2, 2)$, and
	\begin{equation*}
		\frac{6}{5}q_X^{\vee}=c_2(X) \in H^{2,2}(X, \Z).
	\end{equation*}
	Moreover, $\frac{2}{5}q_X^{\vee} \in H^{2,2}(X, \Z)$ is an integral Hodge class of $X$ of type $(2, 2)$.
\end{prop}
If $S$ is a projective K3 surface, then $H^{2,2}(S^{[2]}, \Z)$ is a lattice, where the bilinear form considered is the cup product. We recall \cite[Theorem 8.3]{novario2021hodge}, which gives a basis of the lattice $H^{2,2}(S^{[2]}, \Z)$ when $S$ is general and the Picard group of $S$ is known.
\begin{thm}[N.] \label{thm basis H2,2(S2, Z)}
	Let $S$ be a general projective K3 surface, $\{b_1, \dots, b_r\}$ be a basis of $\Pic(S)$ and $X:=S^{[2]}$.  Then:
	\begin{enumerate}[label=(\roman*)]
		\item $\textnormal{rk}(H^{2,2}(X, \Z))=\frac{(r+1)r}{2}+r+2$.
		\item The following is a basis of the lattice $H^{2,2}(X, \Z)$, which is odd:
		\begin{equation*}
			\left\{b_ib_j, \frac{b_i^2-b_i\delta}{2}, \frac{1}{8}\left( \delta^2+\frac{2}{5}q_X^{\vee}\right), \delta^2\right\}_{1 \le i \le j \le r.}
		\end{equation*}
	In particular, if $S=S_{2t}$ is a generic K3 surface of degree $2t$, and $h \in \Pic(S^{[2]}_{2t})$ is the class induced by the ample generator of $\Pic(S_{2t})$, then
	\begin{equation*}
		H^{2,2}(S^{[2]}_{2t}, \Z)=\Z h^2 \oplus \Z \frac{h^2-h\delta}{2} \oplus \Z \frac{1}{8} \left( \delta^2 + \frac{2}{5}q_X^{\vee}\right) \oplus \Z \frac{2}{5}q_X^{\vee}.
	\end{equation*}
Moreover, $\textnormal{disc}(H^{2,2}(S^{[2]}_{2t}, \Z))=84t^3$ and the Gram matrix in the basis given above is the following:
\begin{equation*}
	\begin{pmatrix}
		12t^2 & 6t^2 & 2t & 20t \\
		6t^2 & t(3t-1) & t & 10t \\
		2t & t & 1 & 9 \\
		20t & 10t & 9 & 92
	\end{pmatrix}.
\end{equation*}
	\end{enumerate}
\end{thm}
Note that in the case of the Hilbert square of a generic K3 surface $S_{2t}$ we have substituted the element $\delta^2$ in a basis of $H^{2,2}(X, \Z)$ with $\frac{2}{5}q_X^{\vee}$: this choice slightly simplifies computations in Section \ref{Section irreducibility property}.
\section{Generalities on Hilbert squares of generic K3 surfaces} \label{section generalities HS general K3 surfaces}
In this section we recall some important results on Hilbert squares of generic K3 surfaces of degree $2t$ and we present the main problem of this paper. Recall that the \emph{nef cone} of a projective variety $Y$ is the convex cone $\Nef(Y) \subseteq \NS(Y)_{\R}$ generated by classes of nef divisors in $\NS(Y)_{\R} = \NS(Y) \otimes \R$, where $\NS(Y)$ is the Néron--Severi group of $Y$, the \emph{moving cone} is the convex cone $\Mov(Y) \subseteq \NS(Y)_{\R}$ generated by classes of divisors whose complete linear system has base locus of codimension greater than $1$, and the \emph{pseudoeffective cone} is the convex cone $\overline{\text{Eff}}(Y) \subseteq \NS(Y)_{\R}$ generated by classes of pseudoeffective divisors. We present the following version of the Bayer--Macrì theorem, see \cite[Proposition 13.1]{bayer2014mmp}, describes the closure of the moving cone, the nef cone and the pseudoeffective cone of the Hilbert square of a generic K3 surface.
\begin{thm}[Bayer--Macrì] \label{thm Bayer Macri}
Let $S_{2t}$ be a projective K3 surface with $\Pic(S_{2t})=\Z H$ and $H^2=2t$. Let $X:=S^{[2]}_{2t}$ be its Hilbert square. Denote by $h \in \Pic(X)$ the line bundle induced by the ample generator of $\Pic(S_{2t})$. Then the cones of classes of divisors can be described as follows.
\begin{enumerate}
	\item The extremal rays of the closure of the moving cone $\Mov(X)$ are spanned by $h$ and $h-\mu_t\delta$, where:
	\begin{itemize}
		\item if $t$ is a perfect square, $\mu_t=\sqrt{t}$;
		\item if $t$ is not a perfect square and $(c, d)$ is the minimal solution of the Pell equation $P_t(1)$, then $\mu_t=t \cdot \frac{d}{c}$.
	\end{itemize}
	\item The extremal rays of the nef cone $\Nef(X)$ are spanned by $h$ and $h-\nu_t\delta$, where:
	\begin{itemize}
		\item if the equation $P_{4t}(5)$ is not solvable, $\nu_t=\mu_t$;
		\item if the equation $P_{4t}(5)$ is solvable and $(a_5, b_5)$ is its minimal solution, $\nu_t=2t \cdot \frac{b_5}{a_5}$.
	\end{itemize}
	\item The extremal rays of the pseudoeffective cone $\overline{\textnormal{Eff}}(X)$ are spanned by $\delta$ and $h-\omega_t\delta$, where:
	\begin{itemize}
		\item if $t$ is a perfect square, $\omega_t=\sqrt{t}$;
		\item if $t$ is not a perfect square, $\omega_t=\frac{c}{d}$, where $(c, d)$ is the minimal solution of $P_t(1)$.
	\end{itemize}
\end{enumerate} 
\end{thm}
Let $X$ be as in Theorem \ref{thm Bayer Macri}. We recall the description of the group of biregular automorphisms $\Aut(X)$ given by Boissière, Cattaneo, Nieper-Wi{\ss}kirchen and Sarti in \cite{boissiere2016automorphism}.
\begin{thm}[Proposition 4.3, Proposition 5.1, Lemma 5.3, Theorem 5.5 in \cite{boissiere2016automorphism}] \label{thm BCNS Aut(X)}
Let $S_{2t}$ be a projective K3 surface with $\Pic(S_{2t})= \Z H$ and $H^2=2t$. Let $h \in \Pic(S^{[2]}_{2t})$ be the line bundle induced by the ample generator of the Picard group $\Pic(S_{2t})$.
\begin{enumerate}
	\item If $t=1$, then $S_2$ is the double cover of $\Pn^2$ branched along a smooth sextic curve and if $\iota$ is the covering involution then
	\begin{equation*}
		\Aut(S^{[2]}_2)=\{\textnormal{id}_{S^{[2]}_2}, \iota^{[2]}\} \cong \Z/2\Z.
	\end{equation*}
	\item If $t \ge 2$, then $X:=S^{[2]}_{2t}$ admits a non-trivial automorphism if and only if one of the following equivalent conditions is satisfied.
	\begin{itemize}
		\item The integer $t$ is not a square, the Pell-type equation $P_{4t}(5)$ has no solution and the negative Pell equation $P_t(-1)$ has a solution.
		\item There exists an ample class $D \in \NS(S^{[2]}_{2t})$ such that $q_X(D)=2$.
	\end{itemize}
	In this case the class $D$ is unique, the automorphism $\iota$ is unique and it is a non-natural anti-symplectic involution. Its action on $\NS(S^{[2]}_{2t})$ is the reflection in the span of the class $D$ of square $2$ represented in the basis $\{h, -\delta\}$ by the matrix
	\begin{equation} \label{eq matrix action iota}
		\begin{pmatrix}
			c & -d \\
			td & -c
		\end{pmatrix},
	\end{equation}
i.e., $\iota^*(xh-y\delta)=(cx-dy)h-(tdx-cy)\delta$, where $(c, d)$ is the minimal solution of the Pell equation $P_t(1)$.
\end{enumerate}
\end{thm}
Let $X$ be the Hilbert square of a generic K3 surface of degree $2t$, and suppose that $X$ admits an ample class $D \in \Pic(X)$ with $q_X(D)=2$. By Theorem \ref{thm Bayer Macri} the nef cone $\Nef(X)$ and the closure of the moving cone $\Mov(X)$ coincide, and by Theorem \ref{thm BCNS Aut(X)} the action on $\NS(S^{[2]}_{2t})\otimes\R$ of the anti-symplectic involution $\iota$ preserves the nef cone of $S^{[2]}_{2t}$, exchanging the two extremal rays. This is true also for a non-natural non-trivial birational automorphism $\iota \in \Bir(S^{[2]}_{2t})$ on the Hilbert square of a generic K3 surface $S_{2t}$ of degree $2t$, whose action on $\NS(S^{[2]}_{2t})\otimes\R$ preserves the closure of the moving cone exchanging the two extremal rays, see \cite[Lemma 6.22]{markman2011survey}. For a description of $\Bir(S^{[2]}_{2t})$ see \cite[Proposition B.3]{debarre2019period}, for a description of $\Aut(S^{[n]}_{2t})$ for $n \ge 3$ see \cite{cattaneo2019automorphisms} and for a description of $\Bir(S^{[n]}_{2t})$ for $n \ge 3$ see \cite{beri2020birational}.

\bigskip

Let $X=S^{[2]}_{2t}$ be as above: the ample divisor $D$ with $q_X(D)=2$ is unique by Theorem \ref{thm BCNS Aut(X)} and its class in $\Pic(X)$ is $bh-a\delta$, where $h \in \Pic(X)$ is the line bundle induced by the ample generator of $\Pic(S_{2t})$, and $(a, b)$ is the minimal solution of the negative Pell equation $P_t(-1)$. By the Kodaira vanishing theorem and a Riemann--Roch formula for IHS fourfolds of $K3^{[2]}$-type, see \cite[Formula (2.2.7)]{o2010higher}, we have $\Dim(H^0(X, \Ol_X(D)))=6$. Hence the rational map\,\,$\varphi_{|D|}$ induced by the complete linear system $|D|$ has $\Pn^5$ as codomain, i.e., 
\begin{equation*}
	\varphi_{|D|}: X \dashrightarrow \Pn^5.
\end{equation*}
We can now state the main problem of this paper.
\begin{qst*}
	Let $X$ be the Hilbert square of a generic K3 surface of degree $2t$ such that $X$ admits an ample divisor $D$ with $q_X(D)=2$. Describe the base locus of the complete linear system $|D|$ and describe the rational map 
	\begin{equation*}
		\varphi_{|D|}: X \dashrightarrow \Pn^5.
	\end{equation*}
\end{qst*}
As observed in \cite[$\S 1$]{boissiere2016automorphism}, the first values of $t$ such that $X=S^{[2]}_{2t}$ as in Theorem \ref{thm BCNS Aut(X)} has an ample divisor\,\,$D$ with $q_X(D)=2$ are $t=2, 10, 13, \text{etc}$. The case $t=2$ was studied by Welters in \cite{welters1981abel} and Beauville in \cite{beauville1983some}: in this case $S_4 \subset \Pn^3$ is a smooth quartic surface of $\Pn^3$ with Picard rank $1$, in particular $S_4$ does not contain any line, the map $\varphi_{|D|}$ is a finite morphism of degree\,\,$6$ with image $\mathbb{G}(1, \Pn^3) \subset \Pn^5$, the Grassmannian of lines in $\Pn^3$, and $\varphi_{|D|}(x):=l_x$ is the unique line in $\Pn^3$ which contains the support of the subscheme $x$ of $S_4$. Moreover, if $l_x \cap S_4=x \cup x^{\prime}$, where $x, x^{\prime} \in S_4^{[2]}$, then the anti-symplectic involution $\iota \in \Aut(X)$ is the \emph{Beauville involution}, defined as $\iota(x):=x^{\prime}$. Since $S_4$ does not contain any line, $\iota$ is everywhere well defined. Note that $\varphi_{|D|}$ factors through the quotient $\pi: S_4^{[2]} \rightarrow S_4^{[2]}/\langle \iota \rangle$ with respect to the Beauville involution.

A first step in the analysis of the base locus of $|D|$ in the Problem above is given by the following result, which is a corollary of \cite[Theorem 4.7]{riess2018base}.
\begin{prop}
	Let $X$ be the Hilbert square of a projective K3 surface of Picard rank $1$. Consider a big and nef divisor $D \in \Div(X)$. Then $|D|$ has no fixed part, i.e., the base locus of $|D|$ has codimension greater than $1$.
\end{prop}
\begin{proof}
	By \cite[Theorem 4.7]{riess2018base} the complete linear system has a fixed part if and only if $D=mL+F$, where $m \ge 2$, the class of $L$ is movable with $q_X(L)=0$ and $F$ is a reduced and irreducible divisor with $q_X(F) < 0$ and $(L, F)=1$. Since $\Pic(X)= \Z h \oplus \Z \delta$, where $h$ is the line bundle induced by the ample generator of the Picard group of the underlying K3 surface, and $(h, \delta)=0$, there are no elements $L$ and $F$ in $\Pic(X)$ such that $(L, F)=1$. We conclude that $|D|$ has no fixed part.
\end{proof}
We now describe the set $H^{3,3}(S^{[2]}_{2t}, \Z)$ of integral Hodge classes of type $(3, 3)$ on the Hilbert square of a K3 surface\,\,$S_{2t}$ with $\Pic(S_{2t}) = \Z H$ and $H^2=2t$: this will be useful in the proof of Proposition \ref{prop dim Y=3, deg Y}.
\begin{thm} \label{thm H_2(S2, Z)}
	Let $S$ be a K3 surface and suppose that $\{b_1, \dots , b_r\}$ is a basis of $\Pic(S^{[2]})$. Then a basis of the $\Z$-module $H^{3,3}(S^{[2]}, \Z)$ is given by $\{cl(b_1^{\vee}), \dots , cl(b_r^{\vee})\}$, where $b_i^{\vee} \in H_2(S^{[2]}, \Z)$ is the primitive element associated to $b_i \in H^2(S^{[2]}, \Z)$ by Proposition\,\,$\ref{Prop Hassett Tschinkel}$, and $cl(b_i^{\vee})$ is the inverse of the Poincaré dual of $b_i^{\vee}$. In particular, let $X=S_{2t}^{[2]}$ be the Hilbert square of a projective K3 surface with $\Pic(S_{2t})=\Z H$ and $H^2=2t$, and let $h \in \Pic(X)$ be the class induced by the ample generator of $\Pic(S_{2t})$. Consider $h^{\vee}_6:=cl(h^{\vee})$ and $\delta^{\vee}_6:=cl(\delta^{\vee}) \in H^6(X, \Z)$. Then
		\begin{equation*}
			H^{3, 3}(X, \Z) \cong \Z h^{\vee}_6 \oplus \Z \delta^{\vee}_6.
		\end{equation*}
		Moreover, $h^{\vee}_6=\frac{1}{6t}h^3$ and $\delta^{\vee}_6=\frac{1}{4t}h^2\delta$, and the following equalities hold in $H^{3,3}(X, \Z)$:
		\begin{equation*}
			\delta^3=-\frac{3}{t}h^2\delta, \qquad h\delta^2=-\frac{1}{3t}h^3, \qquad q_X^{\vee}h=\frac{25}{6t}h^3, \qquad q_X^{\vee}\delta=\frac{25}{2t}h^2\delta.
		\end{equation*}
\end{thm}
\begin{proof}
The torsion free quotient group $H_2(S^{[2]}, \Z)_f$ is a Hodge structure of weight $-2$. Moreover the cap product
\begin{equation*}
	H^2(S^{[2]}, \Z) \otimes H_2(S^{[2]}, \Z)_f \rightarrow H_0(S^{[2]}, \Z)
\end{equation*}
is a morphism of Hodge structures of weight $0$, see for instance \cite{peters2008mixed}, hence if $R \in H_2(S^{[2]}, \Z)_f$ is primitive of type $(-1, -1)$, the corresponding $\rho \in H^2(S^{[2]}, \Z)$ is of type $(1, 1)$. Thus $\{b_1^{\vee}, \dots , b_r^{\vee}\}$ is a basis of
\begin{equation*}
	H_{-1, -1}(S^{[2]}, \Z) := H_2(S^{[2]}, \Z) \cap H_{-1, -1}(S^{[2]}),
\end{equation*}
where $H_{-1, -1}(S^{[2]})$ is the component of type $(-1, -1)$ of $H_2(S^{[2]}, \C)$. The Poincaré duality
\begin{equation*}
	PD: H^6(S^{[2]}, \Z) \xrightarrow{\sim} H_2(S^{[2]}, \Z)_f
\end{equation*}
is an isomorphism of Hodge structures of weight $-4$, hence $H^{3,3}(S^{[2]}, \Z) \cong H_{-1, -1}(S^{[2]}, \Z)$, so $\{cl(b_1^{\vee}), \dots , cl(b_r^{\vee})\}$ is a basis of $H^{3,3}(S^{[2]}, \Z)$.
	
Let now $X=S^{[2]}_{2t}$ be the Hilbert square of a projective K3 surface with $\Pic(S_{2t})=\Z H$ and $H^2=2t$, and let $h \in \Pic(X)$ be the line bundle induced by the ample generator of $\Pic(S_{2t})$. The cup product with $h^2$ gives an isomorphism between $H^{1,1}(X, \Q)$ and $H^{3,3}(X, \Q)$. Then $H^{3,3}(X, \Q)$ is 2-dimensional and it is generated by $h^3$ and\,\,$h^2\delta$ since $H^{1,1}(X, \Q) \cong \Pic(X)\otimes \Q$. We now represent $\delta^3$, $h\delta^2, q_X^{\vee}h$ and $q_X^{\vee}\delta$ in terms of this basis. If $\delta^3=xh^3+yh^2\delta$ with $x, y \in \Q$, by Proposition $\ref{prop rmk 2.1 OG}$ we have
\begin{equation*}
	h\delta^3=0=x\langle h^2, h^2 \rangle + y \langle h^2, h\delta \rangle=12t^2x, \qquad \delta^4=3\cdot(-2)^2=x\langle h^2, h\delta \rangle + y \langle h^2, \delta^2\rangle=-4ty,
\end{equation*}
	which give $x=0$ and $y=-\frac{3}{t}$, hence
	\begin{equation*}
		\delta^3=-\frac{3}{t} h^2\delta.
	\end{equation*}
	Similarly we obtain
	\begin{equation*}
		h\delta^2=-\frac{1}{3t}h^3, \qquad q_X^{\vee}h=\frac{25}{6t}h^3, \qquad q_X^{\vee}\delta=\frac{25}{2t}h^2\delta.
	\end{equation*}
	Consider now $H^{3,3}(X, \Z)$. The cohomology group $H^6(X, \Z)$ is torsion free, and $H^{3,3}(X, \Q)$ has dimension $2$, hence $\rk(H^{3,3}(X, \Z))=2$. Let $h^{\vee}, \delta^{\vee} \in H_2(X, \Z)$ be the primitive classes which correspond, by Proposition \ref{Prop Hassett Tschinkel}, respectively to the primitive classes $h, \delta \in H^2(X, \Z)$. Using Proposition \ref{Prop Hassett Tschinkel} one gets $h^{\vee}=h$ and $\delta^{\vee}=\frac{\delta}{2}$ seen as elements in $H^2(X, \Q)$. Let $PD$ be the Poincaré duality: then $h_6^{\vee}:=PD^{-1}(h^{\vee})$ and $\delta_6^{\vee}:=PD^{-1}(\delta^{\vee})$ give a basis of $H^{3,3}(X, \Z)$, so we have
	\begin{equation*}
		H^{3,3}(X, \Z) \cong \Z h^{\vee}_6 \oplus \Z \delta^{\vee}_6.
	\end{equation*}
	Moreover, if $\cdot$ denotes the cup product and $\cap$ the cap product, we get
	\begin{equation} \label{eq 1 proof H3,3}
		\displaystyle\int_X h_6^{\vee} \cdot x = \epsilon (h^{\vee} \cap x), \qquad \displaystyle\int_X \delta_6^{\vee} \cdot x = \epsilon (\delta^{\vee} \cap x)
	\end{equation}
	for every $x \in H^2(X, \Z)$, see \cite[p.249]{hatcher2005algebraic}, and $\epsilon: H^0(X, \Z) \xrightarrow{\sim} \Z$. Again by Proposition $\ref{Prop Hassett Tschinkel}$ we have
	\begin{equation} \label{eq 2 proof H3,3}
		\epsilon(h^{\vee} \cap h)=(h, h)=2t, \qquad \epsilon(h^{\vee} \cap \delta)=(h, \delta)= 0.
	\end{equation}
	If we write $h_6^{\vee}=\alpha h^3+\beta h^2\delta$ for some $\alpha, \beta \in \Q$, by Proposition $\ref{prop rmk 2.1 OG}$ we obtain
	\begin{equation} \label{eq 3 proof H3,3}
		\displaystyle\int_X h_6^{\vee} \cdot h=\alpha \langle h^2, h^2 \rangle + \beta \langle h^2, h\delta \rangle = 12 \alpha t^2, \qquad \displaystyle\int_X h_6^{\vee}\cdot \delta = \alpha \langle h^2, h\delta \rangle + \beta \langle h^2, \delta^2 \rangle = -4\beta t.
	\end{equation}
	Then $(\ref{eq 1 proof H3,3})$, $(\ref{eq 2 proof H3,3})$ and $(\ref{eq 3 proof H3,3})$ give $\alpha=\frac{1}{6t}$ and $\beta=0$, hence $h_6^{\vee}=\frac{1}{6t}h^3$. Similarly we obtain $\delta^{\vee}_6=\frac{1}{4t}h^2\delta$.
\end{proof}
\section{Fixed locus of the anti-symplectic involution} \label{Section fixed locus}
Let $X$ be the Hilbert square of a generic K3 surface $S_{2t}$ of degree $2t$ such that $X$ admits an ample divisor $D$ with $q_X(D)=2$. We recall some properties of $F:=\text{Fix}(\iota)$, the locus of points of $X$ fixed by the anti-symplectic involution $\iota \in \Aut(X)$, then we compute its integral Hodge class in $H^{2,2}(X, \Z)$ and we show that the rational map $\varphi_{|D|}$ induced by the complete linear system $|D|$ factors through the quotient $\pi: X \rightarrow X/\langle \iota \rangle$ with respect to $\iota$. As usual we denote by $h \in \Pic(X)$ the line bundle induced by the ample generator of $\Pic(S_{2t})$. Recall that the class of $D$ in $\Pic(X)$ is $bh-a\delta$, where $(a, b)$ is the minimal solution of the negative Pell equation $P_t(-1)$. 

First of all, we prove that $F$ is a connected surface in our case.
\begin{prop} \label{prop F connected}
	Keep notation as above. Then $F=\textnormal{Fix}(\iota)$ is a connected Lagrangian surface in $X$.
\end{prop}
\begin{proof}
	By \cite[Lemma 1]{beauville2011antisymplectic} the fixed locus $F$ is a Lagrangian surface. We now show that $F$ is connected. Let $\mathcal{M}^{\rho}_{\langle 2 \rangle}$ be the moduli space which parametrizes triples $(X, \iota_X, i_X)$, where $X$ is an IHS manifold of $K3^{[2]}$-type, $\iota_X \in \Aut(X)$ is an anti-symplectic involution whose action on $H^2(X, \Z)$ is the reflection in the class of an ample divisor $D$ with $q_X(D)=2$ and $i_X: \langle 2 \rangle \hookrightarrow \NS(X)$ is a primitive embedding such that $i(\langle 2 \rangle)$ contains the class of $D$. Any two points in the moduli space $\mathcal{M}^{\rho}_{\langle 2 \rangle}$ are deformation equivalent by \cite[Corollary 4.1, Theorem 5.1]{boissiere2019non}. Since deformation equivalence preserves topological properties of $\text{Fix}(\iota)$, it suffices to find a point $(X, \iota_X, i_X) \in \mathcal{M}^{\rho}_{\langle 2 \rangle}$ such that $\text{Fix}(\iota_X)$ is connected. Let $X=S^{[2]}_4$ be the Hilbert square of a smooth quartic surface of $\Pn^3$ with Picard rank $1$, and let $\iota_X \in \Aut(X)$ be the Beauville involution: the surface $\text{Fix}(\iota_X)$ is connected by \cite[Corollary 3.4.4]{welters1981abel}. Alternatively, if we show that the class of $D$, which is $bh-a\delta$, has divisibility $1$ in the lattice $(H^2(X, \Z), q_X)$, then the surface $\text{Fix}(\iota)$ is connected by \cite[Main Theorem]{flapan2020antisymplectic}. Since $(a, b)$ is the minimal solution of the negative Pell equation $P_t(-1)$, the integers\,\,$a$ and $b$ are coprime and $b$ is odd by Proposition \ref{prop D=bh-adelta then b odd}. Since $H^2(S_{2t}, \Z)$ is unimodular, there exists $x \in H^2(X, \Z)$ such that $(h, x)=1$, hence there exist $\alpha, \beta \in \Z$ such that $(D, \alpha \delta+\beta x)=1$. We conclude that $\text{div}(bh-a\delta)=1$, and $F$ is connected. 
\end{proof}
Keep notation as above. By \cite[$\S 3$]{camere2019calabi} the quotient variety $X/\langle \iota \rangle$ is singular along $\pi^{\prime}(F)$, where $\pi^{\prime}: X \rightarrow X/\langle \iota \rangle$ is the quotient map. The desingularization of $X/\langle \iota \rangle$, which we call $W$, is the blow-up of $X / \langle \iota \rangle$ along its singular locus, and it is a Calabi--Yau variety. Equivalently, consider the blow-up $\text{Bl}_F(X)$ of $X$ in the locus $F=\text{Fix}(\iota)$ of points fixed by $\iota$. The involution $\iota$ gives rise to an involution $\tilde{\iota}$ on $\text{Bl}_F(X)$ which fixes the exceptional divisor $E \subset \text{Bl}_F(X)$. One can show that the quotient $\text{Bl}_F(X)/\langle \tilde{\iota}\rangle$ is isomorphic to $W$, obtaining the following commutative diagram, see \cite[Theorem 3.6]{camere2019calabi} for more details:
\begin{equation} \label{eq diagram Calabi-Yau}
	\begin{tikzcd}
		\text{Bl}_F(X) \arrow[r, "\pi"]  \arrow[d, "\beta"] & W \arrow[d, "\beta^{\prime}"] \\ 
		X \arrow[r, "\pi^{\prime}"] & X/\langle \iota \rangle \, .
	\end{tikzcd}
\end{equation}
We now state the following useful technical lemma. See \cite[Lemma 4.3]{derenthal2015cox} for a more general statement and for the proof.
\begin{lem} \label{lem Derenthal Laface}
Let $\pi: X \rightarrow Y$ be a double cover of a smooth projective variety such that the branch locus is a smooth prime divisor $B \in \Div(Y)$, and denote by $\iota \in \Aut(X)$ the involution associated to the double cover. If $\Pic(X)^{\iota}$ is the subgroup of $\iota$-invariant line bundles on $X$, then $\pi^*\Pic(Y) \cong \Pic(X)^{\iota}$.
\end{lem}
Similarly to the case in \cite[$\S 2.6$]{van2007nikulin}, since $\iota^*D \cong D$ there is an involution induced by $\iota$ on $\Pn(H^0(X, \Ol_X(D))^{\vee})$ which has two fixed spaces $\Pn^a$ and $\Pn^b$, where $a+1+b+1=6$ and $a=-1$ if the corresponding eigenspace of $\iota^*$ on $H^0(X, \Ol_X(D))$ is zero, similarly for $b$. See \cite[p.32]{mumford1994geometric} for details on the action of $\iota$ on $H^0(X, \Ol_X(D))$. In particular the rational map $\varphi_{|D|}$ induced by the complete linear system $|D|$ factors through the quotient $X \rightarrow X/\langle \iota \rangle$ if and only if the action induced by $\iota$ on $\Pn(H^0(X, \Ol_X(D))^{\vee})$ is trivial.

Let $B \in \Div(W)$ be the branch divisor of $\pi: \text{Bl}_F(X) \rightarrow W$. By \cite[Lemma I.17.1]{barth2015compact} there exists a divisor $N \in \Div(W)$ such that $\Ol_W(2N) = \Ol_W(B)$ in $\Pic(W)$. Let $D$ be the ample divisor on $X$ with $q_X(D)=2$. We show the following result, similar to \cite[Proposition 2.7, Item (2)]{van2007nikulin} obtained in the case of K3 surfaces admitting a symplectic involution.
\begin{prop} \label{prop decomposition H^0(X, D)}
	Keep notation as above. Consider diagram $(\ref{eq diagram Calabi-Yau})$. Let $\mathcal{D}:=\Ol_X(D)$ and $\mathcal{N}:=\Ol_W(N)$. There exists a line bundle $\mathcal{D}_W \in \Pic(W)$ such that $\pi^*\mathcal{D}_W \cong \beta^*\mathcal{D}$. Moreover, the vector space $H^0(X, \mathcal{D})$ decomposes as
	\begin{equation} \label{eq decomposition H^0(X, D)}
		H^0(X, \mathcal{D}) \cong H^0(W, \mathcal{D}_W) \oplus H^0(W, \mathcal{D}_W-\mathcal{N}),
	\end{equation}
	which is the decomposition of $H^0(X, \mathcal{D})$ into $\iota^*$-eigenspaces.
\end{prop}
\begin{proof}
	Consider $\pi: \text{Bl}_F(X) \rightarrow W$ appearing in diagram (\ref{eq diagram Calabi-Yau}): by Lemma \ref{lem Derenthal Laface} we obtain a divisor $D_W \in \Div(W)$ whose class $\mathcal{D}_W$ in $\Pic(W)$ is such that $\pi^*\mathcal{D}_W \cong \beta^*\mathcal{D}$ in $\Pic(\text{Bl}_F(X))$. Note that Lemma \ref{lem Derenthal Laface} can be applied: $F \subset X$ is smooth and connected by Proposition \ref{prop F connected}, so the exceptional divisor of $\beta$ and the branch divisor of $\pi$ are smooth prime divisors.
	Since $\beta$ is a birational morphism, we have $H^0(X, \mathcal{D}) \cong H^0(\text{Bl}_F(X), \beta^*\mathcal{D})$, which is isomorphic to $H^0(\text{Bl}_F(X), \pi^*\mathcal{D}_W)$, being $\beta^*\mathcal{D} = \pi^*\mathcal{D}_W$. Moreover,
	\begin{equation*}
		\begin{array}{lll}
			\pi_*(\pi^*\mathcal{D}_W) & \cong & \pi_*(\pi^*\mathcal{D}_W \otimes \Ol_{\text{Bl}_F(X)}) \\
			& \cong & \mathcal{D}_W \otimes \pi_*\Ol_{\text{Bl}_F(X)},
		\end{array}
	\end{equation*}
	where in the second isomorphism we use the projection formula. By \cite[Lemma I.17.2]{barth2015compact} we have
	\begin{equation*}
		\pi_*\Ol_{\text{Bl}_F(X)} \cong \Ol_W \oplus \Ol_W(-N),
	\end{equation*}
	where $B \in \Div(W)$ is the branch divisor and $N \in \Div(W)$ is the divisor such that $\Ol_W(2N) = \Ol_W(B)$ in $\Pic(W)$. Thus we obtain the isomorphism
	\begin{equation*}
		\pi_*(\pi^*\mathcal{D}_W) \cong \mathcal{D}_W \otimes (\Ol_W \oplus \Ol_W(-N)) \cong \mathcal{D}_W \oplus (\mathcal{D}_W-\mathcal{N}).
	\end{equation*}
	Hence decomposition (\ref{eq decomposition H^0(X, D)}) holds. Moreover, this is the decomposition of $H^0(X, \mathcal{D})$ in $\iota^*$-eigenspaces. Indeed, two global sections $s, t \in H^0(X, \mathcal{D})$ are in the same eigenspace if and only if the rational function $f=s/t$ is $\iota$-invariant, and this is true when the sections belong both to $H^0(W, \mathcal{D}_W)$ or $H^0(W, \mathcal{D}_W-\mathcal{N})$ in the decomposition, hence each of these two spaces is contained in an eigenspace of $H^0(X, \mathcal{D})$. We conclude that $H^0(W, \mathcal{D}_W)$ and $H^0(W, \mathcal{D}_W-\mathcal{N})$ are isomorphic to the two eigenspaces by decomposition (\ref{eq decomposition H^0(X, D)}).
\end{proof}
We then see from Proposition \ref{prop decomposition H^0(X, D)} that the action induced by $\iota$ on $\Pn(H^0(X, \Ol_X(D))^{\vee})$ is trivial, i.e., $\varphi_{|D|}$ factors through the quotient $X \rightarrow X/\langle \iota \rangle$, if and only if either $H^0(X, \mathcal{D}_W)$ or $H^0(W, \mathcal{D}_W-\mathcal{N})$ is zero. We can now prove the main theorem of this section, which is the first step of the geometrical description of the rational map $\varphi_{|D|}$ induced by the complete linear system $|D|$.
\begin{thm} \label{thm class F and comm diagram}
	Let $X=S^{[2]}_{2t}$ be the Hilbert square of a generic K3 surface of degree $2t$. Suppose that $X$ admits an ample divisor $D$ with $q_X(D)=2$. Let $\iota: X \rightarrow X$ be the anti-symplectic involution which generates $\Aut(X)$ and $F:=\textnormal{Fix}(\iota)$ be the locus of points of $X$ fixed by $\iota$. Then
	\begin{equation} \label{eq Hodge class F}
		[F]=5D^2-\frac{2}{5}q_X^{\vee} \in H^{2,2}(X, \Z),
	\end{equation}
	where $[F]$ denotes the fundamental cohomological class of $F$ in $H^{2,2}(X, \Z)$. Moreover, let $\varphi_{|D|}: X \dashrightarrow \Pn^5$ be the rational map induced by the complete linear system $|D|$. Then $\varphi_{|D|}$ factors through the quotient $X \rightarrow X/\langle \iota \rangle$, i.e., the following diagram is commutative:
	\begin{equation} \label{eq comm diagram}
		\begin{tikzcd}
			X \arrow[rr, dashed, "\varphi_{|D|}"]  \arrow[dr] & & \Pn^5\, \\ 
			& X/\langle \iota \rangle. \arrow[ur, dashed] &
		\end{tikzcd}
	\end{equation}
\end{thm}
\begin{proof}
	By Proposition \ref{prop decomposition H^0(X, D)} there exists $\mathcal{D}_W \in \Pic(W)$ such that $\pi^*\mathcal{D}_W \cong \beta^*\mathcal{D}$, using the notation of diagram (\ref{eq diagram Calabi-Yau}). Applying \cite[Proposition 7.3]{camere2019calabi} we have
	\begin{equation*}
		\Dim(H^0(W, \mathcal{D}_W))=\frac{7}{2}+\frac{1}{16}(D|_F)^2,
	\end{equation*}
	hence by decomposition (\ref{eq decomposition H^0(X, D)}) and $\Dim(H^0(X, \mathcal{D}))=6$ we get $\Dim(H^0(W, \mathcal{D}_W)) \in \{0, 1, \dots, 6\}$. Thus we obtain the following possible values for $(D|_F)^2$:
	\begin{equation} \label{eq poss h0DW}
		\begin{array}{lllllllll}
			\Dim(H^0(W, \mathcal{D}_W))=0 & \Leftrightarrow & (D|_F)^2=-56, & & & & \Dim(H^0(W, \mathcal{D}_W))=4 &  \Leftrightarrow & (D|_F)^2=8, \\[1.2ex]
			\Dim(H^0(W, \mathcal{D}_W))=1 & \Leftrightarrow & (D|_F)^2=-40, & & & & \Dim(H^0(W, \mathcal{D}_W))=5 & \Leftrightarrow & (D|_F)^2=24, \\[1.2ex]
			\Dim(H^0(W, \mathcal{D}_W))=2 & \Leftrightarrow & (D|_F)^2=-24, & & & & \Dim(H^0(W, \mathcal{D}_W))=6 & \Leftrightarrow & (D|_F)^2=40.\\[1.2ex]
			\Dim(H^0(W, \mathcal{D}_W))=3 & \Leftrightarrow & (D|_F)^2=-8,
		\end{array}
	\end{equation}
	Since $D$ is ample, $(D|_F)^2>0$ by the Nakai--Moishezon criterion. This implies that $\Dim(H^0(W, \mathcal{D}_W)) \in \{4, 5, 6\}$. We show that $\Dim(H^0(W, \mathcal{D}_W))=6$ by computing $[F] \in H^{2,2}(X, \Z)$, where $[F]$ is the fundamental cohomological class of the fixed locus $F=\text{Fix}(\iota)$ of the involution $\iota$. Let $h \in \Pic(X)$ be the line bundle induced by the ample generator of $\Pic(S_{2t})$. We can write, see for instance \cite[Proposition 7.1]{novario2021hodge},
	\begin{equation} \label{eq proof cl(F)}
		[F]=xh^2+yh\delta+z\delta^2+w \frac{2}{5}q_X^{\vee} \in H^{2,2}(X, \Z),
	\end{equation}
	with $x, y, z, w \in \Q$ to determine. We write $D$ also for its class in $\Pic(X)$: then $D=bh-a\delta$, with $(a, b)$ minimal solution of the Pell-type equation $P_t(-1)$. We denote by $\langle \, \cdot \, , \, \cdot \, \rangle$ the bilinear form of $H^4(X, \Z)$ of Proposition \ref{prop rmk 2.1 OG}. We have the following four conditions.  
	\begin{enumerate}
		\item $\langle [F], (\sigma + \bar{\sigma})^2\rangle =0$, where $\sigma$ is the symplectic form $\sigma \in H^0(X, \Omega^2_X)$, since by \cite[Lemma 1]{beauville2011antisymplectic} the surface $F$ is Lagrangian. If $\eta:=(\sigma+\bar{\sigma}, \sigma+\bar{\sigma})$, we have:
		\begin{equation} \label{eq sigma + sigma bar}
			\langle h^2, (\sigma+\bar{\sigma})^2\rangle = 2t\eta, \qquad \langle \delta^2, (\sigma +\bar{\sigma})^2 \rangle = -2 \eta, \qquad \langle h\delta, (\sigma+\bar{\sigma})^2\rangle = 0, \qquad \langle \frac{2}{5}q_X^{\vee}, (\sigma + \bar{\sigma})^2 \rangle = 10\eta.
		\end{equation}
		Note that $\eta \neq 0$, see \cite[Definition 22.10, Theorem 23.14]{gross2012calabi}, so we obtain from $\langle [F], (\sigma+\bar{\sigma})^2 \rangle = 0$ and $(\ref{eq proof cl(F)})$ the following condition:
		\begin{equation} \label{cond 1 cl(F)}
			tx-z+5w=0.
		\end{equation}
		\item Applying \cite[Theorem 2]{beauville2011antisymplectic} in our case we have $c_2(F)=192$, i.e., $\langle [F], [F] \rangle =192$. This gives, together with (\ref{eq proof cl(F)}), the following condition:
		\begin{equation} \label{cond 2 cl(F)}
			3t^2x^2-ty^2+3z^2+23w^2-2txz+10txw-10zw=48.
		\end{equation}
		\item Recall that by Theorem \ref{thm BCNS Aut(X)} the action of $\iota$ on $\Pic(X) \otimes \Q \cong H^{1,1}(X, \Q)$ is described in the basis $\{h, -\delta\}$ by the matrix (\ref{eq matrix action iota}). Consider the action induced by $\iota$ on $H^{2,2}(X, \Q)$, described in the basis $\{h^2, h\delta, \delta^2, q_X^{\vee}\}$ by 
		\begin{equation} \label{eq def iota^2}
			\iota^*(h^2)=\iota^*h \cdot \iota^*h, \qquad \iota^*(h\delta)=\iota^*h \cdot \iota^*\delta, \qquad \iota^*(\delta^2)=\iota^*\delta \cdot \iota^*\delta, \qquad \iota^*(q_X^{\vee})=q_X^{\vee},
		\end{equation}
		where $\cdot$ denotes the cup product: the last equation comes from the equality $c_2(X)=\frac{6}{5}q_X^{\vee}$ of Proposition \ref{prop O'Grady c_2(X)} and the fact that $\iota$ is an automorphism. Since $F$ is the locus of points fixed by $\iota$, we have $\iota^*([F])=[F]$, i.e., if
		\begin{equation*}
			\iota^*([F])=\tilde{x}h^2+\tilde{y}h\delta+\tilde{z}\delta^2+\frac{2}{5}\tilde{w}q_X^{\vee},
		\end{equation*}
		with $\tilde{x}, \tilde{y}, \tilde{z}, \tilde{w} \in \Q$, then $x=\tilde{x}$, $y=\tilde{y}$, $z=\tilde{z}$, $w=\tilde{w}$. Imposing $x=\tilde{x}$ we obtain the following condition:
		\begin{equation} \label{cond 3 cl(F)}
			x-c^2x-cdy-d^2z=0,
		\end{equation}
		where $(c, d)$ is the minimal solution of the Pell equation $P_t(1)$. One remarks a posteriori that $y=\tilde{y}$, $z=\tilde{z}$ and $w=\tilde{w}$ give the same condition (\ref{cond 3 cl(F)}).
		\item By $(D|_F)^2 \in \{8, 24, 40\}$ we have $\langle [F], D^2 \rangle  \in \{8, 24, 40\}$, since the bilinear form $\langle \, \cdot \, , \, \cdot \, \rangle$ represents the intersection form by Proposition \ref{prop rmk 2.1 OG}. This gives, together with (\ref{eq proof cl(F)}), the following possibilities:
		\begin{equation} \label{cond 4 cl(F)}
			\begin{array}{lll}
				(t+2t^2b^2)x+2abty+(2a^2-1)z+5w=2 & \iff & \langle [F], D^2 \rangle = 8,\\
				(t+2t^2b^2)x+2abty+(2a^2-1)z+5w=6 & \iff & \langle [F], D^2 \rangle = 24, \\
				(t+2t^2b^2)x+2abty+(2a^2-1)z+5w=10 & \iff & \langle [F], D^2 \rangle = 40.
			\end{array}
		\end{equation}
	\end{enumerate}
	If $\langle [F], D^2 \rangle = 8$, which is equivalent to $\Dim(H^0(W, \mathcal{D}_W))= 4$ by (\ref{eq poss h0DW}), the system given by (\ref{cond 1 cl(F)}), (\ref{cond 2 cl(F)}), (\ref{cond 3 cl(F)}) and the first condition in (\ref{cond 4 cl(F)}) has solutions with $x, y, z, w \not \in \Q$, which is impossible. Similarly we cannot have $\langle [F], D^2 \rangle=24$, i.e., $\Dim(H^0(W, \mathcal{D}_W))=5$. We conclude that $\langle [F], D^2 \rangle = 40$, which is equivalent to $\Dim(H^0(W, \mathcal{D}_W))=6$. With the help of a computer, the system given by (\ref{cond 1 cl(F)}), (\ref{cond 2 cl(F)}), (\ref{cond 3 cl(F)}) and the third condition in (\ref{cond 4 cl(F)}) implies $w \in \{-\frac{13}{12}, -1\}$. By Theorem \ref{thm basis H2,2(S2, Z)} we cannot have $w=-\frac{13}{12}$, hence $w=-1$. Imposing $w=-1$, we necessarily obtain only one admissible solution, which is the following:
	\begin{equation*}
		\left\{ \arraycolsep=1pt\def\arraystretch{1.2} \begin{array}{l}
			x=5b^2,\\y=-10ab,\\z=5a^2,\\w=-1. \end{array}\right.
	\end{equation*}
	We conclude that
	\begin{equation*}
		[F]=5D^2-\frac{2}{5}q_X^{\vee} \in H^{2,2}(X, \Z).
	\end{equation*}
	Moreover, we have obtained $H^0(X, \mathcal{D}) \cong H^0(W, \mathcal{D}_W)$, so $H^0(W, \mathcal{D}_W-\mathcal{N})=\{0\}$, which shows that the action induced by $\iota$ on $\Pn(H^0(X, \mathcal{D})^{\vee})$ is trivial, i.e., $\varphi_{|D|}$ factors through the quotient $X \rightarrow X/\langle \iota \rangle$.
\end{proof}
\begin{rmk}
	Ferretti in \cite[Lemma 4.1]{ferretti2012chow} obtained the same relation of (\ref{eq Hodge class F}) in the Chow ring of a smooth double EPW sextic $X$: in that case $F$ is the branch locus of the double cover $f: X \rightarrow Y \subset \Pn^5$, where $Y$ is an EPW sextic.
\end{rmk}
\section{Irreducibility property} \label{Section irreducibility property}
Let $X$ be the Hilbert square of a generic K3 surface of degree $2t$ such that $X$ admits an ample divisor $D$ with $q_X(D)=2$. In this section we prove the \emph{irreducibility property} when $t \neq 2$, which will be, together with Theorem \ref{thm class F and comm diagram}, fundamental to obtain the main result of this paper. This can be seen as an analogue of the irreducibility property in \cite[Proposition 4.1.(1)]{o2008irreducible}. First of all we show that every divisor in the complete linear system $|D|$ is a prime divisor.
\begin{prop} \label{prop every divisor in |D| is prime}
Keep notation as above. Then every divisor $D^{\prime} \in |D|$ is reduced and irreducible.
\end{prop}
\begin{proof}
	By abuse of notation, we write $D$ for an effective divisor $D^{\prime} \in |D|$. Since $q_X(D)=2$, the divisor $D$ is reduced, i.e., it is not of the form $D= \alpha E$ with $\alpha \in \Z$, $\alpha \neq \pm 1$, and $E \in \Div(X)$. Suppose that 
	\begin{equation} \label{D=D1+D2}
		D= D_1+D_2
	\end{equation}
	where $D_1=\sum_i n_iD_{1,i}$ and $D_2=\sum_j m_jD_{2,j}$ are effective divisors with $D_{1, i}, D_{2, j}$ prime divisors which are pairwise distinct, and $n_i, m_j \in \Z_{>0}$. Without loss of generality we can assume that $D_1$ has only one component.  By abuse of notation, we write $D_i$ and $D_{i,j}$ also for their classes in $H^2(X, \Z)$: we show that (\ref{D=D1+D2}) seen in $H^2(X, \Z)$ gives a contradiction. We have
	\begin{equation} \label{eq proof D reduced irr}
		q_X(D)=2=q_X(D_1)+q_X(D_2)+2(D_1, D_2).
	\end{equation}
	Since $D_1$ and $D_2$ are effective divisors with no common components, we can apply \cite[Proposition $4.2.(i)$]{boucksom2004divisorial}, obtaining $(D_1, D_2) \ge 0$. Note that we cannot have neither $q_X(D_1)=0$ nor $q_X(D_2)=0$. Indeed, the equation $q_X(xh-y\delta)=0$ has a non zero solution $(x, y) \in \Z^2$ only when $t$ is a perfect square. If $t$ was a perfect square, the Pell-type equation $P_t(-1)$ would be solvable only for $t=1$, which is a case that we are not considering. Hence $t$ is not a perfect square and $q_X(D_1)\neq 0$, $q_X(D_2) \neq 0$.
	
	Since $H^2(X, \Z)$ is an even lattice, we have only two possibilities: either one between $D_1$ and $D_2$ is zero, or (at least) one of the two has negative square with respect to the BBF form. We show that $q_X(D_i) > 0$ for $i=1, 2$. Assume by contradiction that at least one between $D_1$ and $D_2$ has negative square. This can happen only if there exists a component $D_{1,i}$ or $D_{2, j}$ whose square with respect to the BBF form is negative. Without loss of generality, we can suppose that $q_X(D_{1, 1})<0$. Recall that $H^2(X, \Z) \cong H^2(S_{2t}, \Z) \oplus \langle -2\rangle$, where $S_{2t}$ is a K3 surface, and $H^2(S_{2t}, \Z)$ is a unimodular lattice. Then the divisibility in $H^2(X, \Z)$ of a primitive class can be only $1$ or $2$. Hence we have either $q_X(D_{1,1})=-2$ or $q_X(D_{1,1})=-4$ by \cite[Lemma 3.7]{markman2013prime}, see also \cite[Lemma 3.5]{riess2018base}. We show that $q_X(D_{1,1})=-4$ is not possible. Indeed, if $t=2$, the class of $D_{1,1}$ in $\NS(X)$ is necessarily $h-2\delta$, which is outside the pseudoeffective cone, whose extremal rays are generated by $\delta$ and $2h-3\delta$ by Theorem \ref{thm Bayer Macri}, obtaining a contradiction. If $t \neq 2$, there exists a $(-4)$-class if and only if the Pell-type equation $P_t(2)$ is solvable: since by assumption $P_t(-1)$ has solutions, by \cite[p.106-109]{perron2013lehre}, see also \cite[Proposition 1]{yokoi1994solvability}, we have that $P_t(2)$ has no solution, hence there are no $(-4)$-classes. We conclude that $D_{1,1}$ must be a $(-2)$-class, hence it is either $\delta$ or $\iota^*\delta$, where $\iota$ is the anti-symplectic involution on\,\,$X$ which generates $\Aut(X)$, see Theorem \ref{thm BCNS Aut(X)}. Since $\iota^*D=D$, it is enough to show that $D_{1,1}=\delta$ is not possible. If $D_{1,1}=\delta$, from (\ref{D=D1+D2}) we get $D=n_1\delta+ D_2$, where $n_1 \ge 1$ since $D_1$ is effective by assumption, hence
	\begin{equation*}
		D_2=D-n_1\delta=bh-(a+n_1)\delta,
	\end{equation*}
	where as usual $(a, b)$ is the minimal solution of the negative Pell equation $P_t(-1)$. We show that $D_2$ is outside the pseudoeffective cone. By Theorem \ref{thm Bayer Macri} the extremal rays of the pseudoeffective cone are generated by the classes $\delta$ and $\iota^*\delta=dh-c\delta$, where $(c, d)$ is the minimal solution of the Pell equation $P_t(1)$. Then $\iota^*\delta=2abh-(a^2+tb^2)\delta$ by (\ref{eq relation Pt(1) and Pt(-1)}), hence we need to check that $a+n_1 >  \frac{a^2+tb^2}{2a}$ in order to show that $D_2$ is outside the pseudoeffective cone. This is true since $a^2-tb^2=-1$. We obtain a contradiction, so $q_X(D_i) \ge 0$ for $i=1, 2$. Moreover, we have already remarked that $q_X(D_i)\neq 0$ for $i=1, 2$, thus $q_X(D_i) > 0$. We get a contradiction with (\ref{eq proof D reduced irr}), so one between $D_1$ and $D_2$ is zero. If $D_2=0$, then $n_1=1$ since $q_X(D)=n_1^2q_X(D_{1,1})=2$, so $D$ is reduced and irreducible. If $D_1=0$, we repeat the argument for $D=D_2$ until we obtain a $D$ which is reduced and irreducible.
\end{proof}
Suppose now that $D_1, D_2 \in |D|$ are two distinct divisors. We want to study the surface $D_1 \cap D_2$ and see if it is reduced and irreducible. We first show what happens when $t=2$, i.e., when $X=S^{[2]}_4$ is the Hilbert square of a smooth quartic surface of $\Pn^3$ with Picard rank $1$, a case already studied in \cite{welters1981abel} and \cite{beauville1983some}. We briefly recall some notation from \cite[$\S 1.5$, $\S 6.2$]{griffiths1978principles} on Schubert varieties of $\mathbb{G}(1, \Pn^3)$. Fix a complete flag in $\Pn^3$, i.e., let $v_0$ be a point in $\Pn^3$, let $L_0$ be a line in $\Pn^3$ and let $H_0$ be a plane in $\Pn^3$ such that $v_0 \in L_0 \subset H_0$. Consider the following two Schubert varieties of dimension $2$ of the Grassmannian $\mathbb{G}(1, \Pn^3)$,
\begin{equation} \label{eq Schubert varieties}
	\Sigma_{1,1}:=\{L \in \mathbb{G}(1, \Pn^3)\, | \, L \subset H_0\}, \qquad \Sigma_2:=\{L \in \mathbb{G}(1, \Pn^3)\, | \, v_0 \in L\},
\end{equation}
and denote by $\sigma_{1,1}$ and $\sigma_2$ the corresponding classes in $H^4(\mathbb{G}(1, \Pn^3), \Z)$. Recall that $\mathbb{G}(1, \Pn^3)$ can be embedded in\,\,$\Pn^5$ as a quadric. Then $\sigma_{1,1}+\sigma_2=\sigma_1^2$, where $\sigma_1$ is the hyperplane class of $\Pn^5$ restricted to $\mathbb{G}(1, \Pn^3)$, moreover $\int_{\mathbb{G}(1, \Pn^3)}\sigma_{1,1} \cdot \sigma_2=0$ and $\int_{\mathbb{G}(1, \Pn^3)} \sigma_{1,1}^2=\int_{\mathbb{G}(1, \Pn^3)} \sigma_2^2=1$. Recall that $S_4$ does not contain any line, since it has Picard rank $1$. Let $f:=\varphi_{|D|}$ be the map seen in Section \ref{section generalities HS general K3 surfaces}, i.e.,
\begin{equation} \label{eq def f t=2}
	f: S_4^{[2]} \rightarrow \mathbb{G}(1, \Pn^3) \subset \Pn^5, \qquad x \mapsto l_x,
\end{equation}
where $l_x$ is the unique line in $\Pn^3$ which contains the support of the subscheme $x$. Moreover, the class of $D$ is $h-\delta$, where as usual $h \in \Pic(S_4^{[2]})$ is the class induced by the ample generator of $\Pic(S_4)$. Since $f$ is the map induced by the complete linear system $|D|$, we have $f^*\sigma_1=c_1(\Ol_X(D))$, hence $f^*(\sigma_{1,1}+\sigma_2)=c_1(\Ol_X(D))^2$. This shows that, given two distinct divisors $D_1, D_2 \in |D|$, the surface $D_1 \cap D_2$ can be reducible, and if so, we have $[D_1 \cap D_2]=A+B \in H^{2,2}(X, \Z)$, where $A, B \in H^{2,2}(X, \Z)$ are the effective classes of the two irreducible components of $D_1 \cap D_2$. Moreover, $A=f^*\sigma_{1,1}$ and $B=f^*\sigma_2$. If $S_4$ is generic it is possible to compute the classes of $f^*\sigma_{1,1}$ and $f^*\sigma_2$, obtaining the following equalities:
\begin{equation} \label{eq pullback Schubert classes}
f^*\sigma_{1,1}=\frac{1}{2}h^2-\frac{1}{4}\delta^2-\frac{1}{2}h\delta-\frac{1}{10}q_X^{\vee}, \qquad f^*\sigma_2=\frac{1}{2}h^2+\frac{5}{4}\delta^2-\frac{3}{2}h\delta+\frac{1}{10}q_X^{\vee}.
\end{equation}
See \cite[Proposition 4.4.4]{novario2021ths} for details. It is also possible to geometrically describe the surface $F \subset S_4^{[2]}$ of points fixed by the Beauville involution. We call \emph{bitangent} of $S_4 \subset \Pn^3$ a line of $\Pn^3$ which intersects $S_4$ either in two points both with multiplicity $2$ or in a point with multiplicity\,\,$4$. Let $\text{Bit}(S_4) \subset \mathbb{G}(1, \Pn^3)$ be the surface of bitangents of $S_4 \subset \Pn^3$. Then we have $f|_F: F \xrightarrow{\sim} \text{Bit}(S_4)$. If $S_4$ is generic, starting from $[\text{Bit}(S_4)]=12\sigma_2+28\sigma_{1,1} \in H^{2,2}(\mathbb{G}(1, \Pn^3), \Z)$, see for instance \cite[Proposition 3.3]{arrondo2001focus}, it is possible to show that $f^*[\text{Bit}(S_4)]=20h^2+8\delta^2-32h\delta-\frac{8}{5}q_X^{\vee} \in H^{2,2}(S_4^{[2]}, \Z)$, see \cite[Proposition 4.4.4]{novario2021ths}.

\bigskip

We come back to the general problem of the Hilbert square $X$ of a generic K3 surface admitting an ample divisor $D \in \Div(X)$ with $q_X(D)=2$. Before stating the irreducibility property, we show the following technical lemma.
\begin{lem} \label{lem technical}
Let $X$ and $D$ be as in Theorem $\ref{thm class F and comm diagram}$, and let $D_1, D_2 \in |D|$ be two distinct divisors. Denote by $\iota$ the anti-symplectic involution which generates the automorphism group $\Aut(X)$. Suppose that there is no decomposition of the form
\begin{equation*}
	[D_1 \cap D_2]=A+B \in H^{2,2}(X, \Z),
\end{equation*}
where $[D_1 \cap D_2]$ is the fundamental cohomological class of the surface $D_1 \cap D_2$, and $A, B \in H^{2,2}(X, \Z)$ are effective classes such that $\iota^*(A)=A$ and $\iota^*(B)=B$ with $\iota^*$ described in $(\ref{eq def iota^2})$. Then the surface $D_1 \cap D_2$ is reduced and irreducible.
\end{lem}
\begin{proof}
	Suppose that 
	\begin{equation*}
		[D_1 \cap D_2]=A_1+A_2+\dots + A_n \in H^{2,2}(X, \Z),
	\end{equation*} 
	where $A_i \in H^{2,2}(X, \Z)$ are effective classes, not necessarily pairwise distinct. By Theorem \ref{thm BCNS Aut(X)} we have $\iota^*[D]=[D]$, hence $\iota^*[D_1 \cap D_2]=[D_1 \cap D_2]$. If $n>1$ is odd, then there exists $i$ such that $\iota^*A_i=A_i$ since $\iota$ is an involution, hence we take $A:=A_i$ and $B:=A_1+ \dots + A_{i-1}+A_{i+1}+ \dots + A_n$. Thus $\iota^*A=A$ and $\iota^*B=B$. We obtain a contradiction with the assumption that there is no decomposition of this form. Suppose now that $n=2$ and $\iota^*A_1=A_2$. We show that this is not possible. Recall that the class of $D$ in $\Pic(X)$ is $bh-a\delta$, with $(a, b)$ minimal solution of the Pell-type equation $P_t(-1)$. By Theorem \ref{thm basis H2,2(S2, Z)} the classes $A_1$ and $A_2$ in $H^{2,2}(X, \Z)$ are of the form
	\begin{equation*}
		\begin{array}{l}
			A_1=\left( x+\frac{y}{2} \right) h^2+\frac{z}{8} \delta^2 -\frac{y}{2}h\delta+\left( \frac{1}{8}z+w \right) \frac{2}{5}q_X^{\vee}, \\[1ex]
			A_2=\left( b^2-x-\frac{y}{2} \right) h^2+\left( a^2-\frac{z}{8}\right) \delta^2 + \left( -2ab+\frac{y}{2}\right) h\delta-\left( \frac{1}{8}z+w\right) \frac{2}{5}q_X^{\vee},
		\end{array}
	\end{equation*}
	for some $x, y, z, w \in \Z$. By Theorem \ref{thm BCNS Aut(X)} and $(\ref{eq def iota^2})$, we obtain
	\begin{equation*}
		\begin{array}{lll}
			\iota^*A_1 & = & \left( c^2 \left( x + \frac{y}{2} \right) +cd  \left( -\frac{y}{2} \right) +\frac{z}{8}d^2 \right) h^2 \\[0.8ex]
			& & + \left( t^2d^2 \left( x+\frac{y}{2}\right) -cdt\frac{y}{2}+c^2\frac{z}{8}\right) \delta^2 \\[0.8ex]
			& & +\left( -2cdt\left( x+ \frac{y}{2}\right) +c^2\frac{y}{2}+td^2\frac{y}{2}-2cd\frac{z}{8}\right) h\delta \\[0.8ex]
			& & +\left( \frac{1}{8}z+w\right) \frac{2}{5}q_X^{\vee}, 
		\end{array}
	\end{equation*}
	and similarly
	\begin{equation*}
		\begin{array}{lll}
			\iota^*A_2 & = & \left( c^2\left( b^2-x-\frac{y}{2} \right) +cd\left(\frac{y}{2} -2ab \right) +d^2\left( a^2-\frac{z}{8}\right) \right) h^2 \\[0.8ex]
			& & + \left( t^2d^2\left( b^2-x-\frac{y}{2}\right) + \left( \frac{y}{2}-2ab \right)cdt +c^2 \left( a^2-\frac{z}{8} \right) \right) \delta^2 \\[0.8ex]
			& & +\left( -2cdt\left(b^2-x-\frac{y}{2}\right) +(c^2+td^2)\left( 2ab-\frac{y}{2} \right) -2cd\left( a^2-\frac{z}{8} \right) \right) h\delta \\[0.8ex]
			& & -\left( \frac{1}{8}z+w\right) \frac{2}{5}q_X^{\vee}. 
		\end{array}
	\end{equation*}
	Imposing $\iota^*A_1=A_2$, we obtain a system whose solution is
	\begin{equation*}
		\begin{cases}
			x=\frac{1}{2}b^2-ab, \\
			y=2ab, \\
			z=4a^2, \\
			w=-\frac{a^2}{2}.
		\end{cases}
	\end{equation*}
	Recall that $b$ is odd by Proposition \ref{prop D=bh-adelta then b odd}, so $x \not \in \Z$, which is not possible. Since by assumption we cannot have $\iota^*A_1=A_1$, one between $A_1$ and $A_2$ is zero, so $D_1 \cap D_2$ is reduced and irreducible. If $n>1$ is even, if there exists an $i$ such that $\iota^*A_i=A_i$, we proceed as in the case of $n$ odd, otherwise without loss of generality we can assume that $\iota^*A_1=A_2$. Then, taking $A:=A_1+A_2$ and $B:=A_3+ \dots + A_n$, we have $\iota^*A=A$ and $\iota^*B=B$, which contradicts the assumption. We conclude that $D_1 \cap D_2$ is reduced and irreducible.
\end{proof}
We can now state the main theorem of this section, the \emph{irreducibility property}. 
\begin{thm} \label{thm irreducibility property}
	Let $X$ be the Hilbert square of a projective K3 surface $S_{2t}$ with $\Pic(S_{2t})= \Z H$ and $H^2=2t$ and such that $X$ admits an ample divisor $D \in \Div(X)$ with $q_X(D)=2$. Let $D_1, D_2 \in |D|$ be two distinct divisors.
	\begin{thmlist}
		\item If $t=2$, then the surface $D_1 \cap D_2$ can be reducible. If $[D_1 \cap D_2]=A+B$, where $[D_1 \cap D_2]$ is the fundamental cohomological class of $D_1 \cap D_2$ in $H^{2,2}(X, \Z)$, then $A=f^*\sigma_{1,1}$ and $B=f^*\sigma_2$, where $f: X \rightarrow \mathbb{G}(1, \Pn^3)$ is the map given in $(\ref{eq def f t=2})$ and $\sigma_{1,1}$, $\sigma_2$ are the classes in $H^4(\mathbb{G}(1, \Pn^3), \Z)$ of the Schubert varieties in $(\ref{eq Schubert varieties})$. Moreover, if $S_4$ is generic, then the effective classes $A$ and $B$ in $H^{2,2}(X, \Z)$ are
		\begin{equation*}
			A=\frac{1}{2}h^2-\frac{1}{4}\delta^2-\frac{1}{2}h\delta-\frac{1}{10} q_X^{\vee}, \qquad B=\frac{1}{2}h^2+\frac{5}{4}\delta^2-\frac{3}{2}h\delta+\frac{1}{10}q_X^{\vee}.
		\end{equation*} \label{thm irr property t=2}
		\item If $t \neq 2$ and $S_{2t}$ is generic, then $D_1 \cap D_2$ is a reduced and irreducible surface. \label{thm irr property t neq 2}
	\end{thmlist}
\end{thm}
\begin{proof}
	We begin with $t=2$, so $X=S_4^{[2]}$ is the Hilbert square of a smooth quartic surface of $\Pn^3$ with Picard rank\,\,$1$ and the class of  $D$ is $h-\delta \in \Pic(X)$, where $h$ is the class induced by the ample generator of $\Pic(S_4)$. For the reader's convenience, we work under the assumption that $S_4$ is generic: the procedure that we will develop in this way is exactly the same that we will use for Item $(ii)$, where $S_{2t}$ is generic by assumption, but easier to follow for the case $t=2$. By Lemma $\ref{lem technical}$, the surface $D_1 \cap D_2$ can be reducible only if there exist effective classes $A, B \in H^{2,2}(X, \Z)$ such that $[D_1 \cap D_2]=A+B$ and $\iota^*(A)=A$, $\iota^*(B)=B$, where $\iota$ is the Beauville involution. Since $(h-\delta)^2=h^2+\delta^2-2h\delta$, by Theorem \ref{thm basis H2,2(S2, Z)} we can write, for some $x, y, z, w \in \Z$,
	\begin{equation*}
		\begin{array}{l}
			A=\left( x+\frac{y}{2} \right)h^2 +\frac{z}{8} \delta^2 -\frac{y}{2}h\delta+\left( \frac{1}{8}z+w \right) \frac{2}{5}q_X^{\vee} \in H^{2,2}(X, \Z), \\[1ex]
			B=\left( 1-x-\frac{y}{2} \right) h^2+\left( 1-\frac{z}{8}\right) \delta^2 + \left(\frac{y}{2} -2 \right) h\delta-\left( \frac{1}{8}z+w\right) \frac{2}{5}q_X^{\vee} \in H^{2,2}(X, \Z).
		\end{array}
	\end{equation*}
	\begin{itemize}
		\item By assumption $A$ and $B$ are effective. Moreover, $h \in \Pic(X)$ is nef by Theorem \ref{thm Bayer Macri}. By Kleiman's theorem we have $\langle A, h^2 \rangle \ge 0$ and $\langle B, h^2 \rangle \ge 0$, where $\langle \, \cdot \, , \, \cdot \, \rangle$ is the bilinear form of $H^4(X, \Z)$ given in Proposition \ref{prop rmk 2.1 OG}, which coincides with the intersection pairing. We obtain the following condition:
		\begin{equation} \label{Eq1 t=2}
			0 \le 12 x + 6y +z + 10w \le 10.
		\end{equation}
		\item Let $\sigma \in H^0(X, \Omega^2_X)$ be the symplectic form. Then, since $A$ and $B$ are effective classes in $H^{2,2}(X, \Z)$ by assumption, we have
		\begin{equation*}
			\displaystyle\int_A (\sigma+\bar{\sigma})^2=2\displaystyle\int_A \sigma \wedge \bar{\sigma} \ge 0, \qquad \displaystyle\int_B (\sigma+\bar{\sigma})^2=2\displaystyle\int_B \sigma \wedge \bar{\sigma} \ge 0,
		\end{equation*}
		since $\sigma \wedge \bar{\sigma} \in H^{2,2}(X)$ is a volume form on $A$ and on $B$. Note that $\sigma \wedge \bar{\sigma}$ can be zero on $A$ or $B$, for instance when $A$ or $B$ are Lagrangian. Hence $\langle A, (\sigma+\bar{\sigma})^2 \rangle \ge 0$ and $\langle B, (\sigma+\bar{\sigma})^2 \rangle \ge 0$. Using $(\ref{eq sigma + sigma bar})$ we obtain the following condition:
		\begin{equation} \label{Eq2 t=2}
			0 \le 4x+2y+z+10w \le 2.
		\end{equation}
		\item By abuse of notation we write $D$ also for its class in $\Pic(X)$. Since $D$ is ample and by assumption $A$ and\,\,$B$ are effective, by the Nakai--Moishezon criterion we have $\langle A, D^2 \rangle > 0$ and $\langle B, D^2 \rangle > 0$, and we obtain the following condition:
		\begin{equation} \label{Eq3 t=2}
			0< 40x+12y+3z+20w<12.
		\end{equation}
		\item By Theorem \ref{thm BCNS Aut(X)} and $(\ref{eq def iota^2})$ we have
		\begin{equation*}
			\iota^*A=\left( 9x+\frac{3}{2}y+\frac{1}{2}z \right) h^2+\left( 16x+2y+\frac{9}{8} \right) \delta^2-\left( 24x+\frac{7}{2}y+\frac{3}{2}z \right) h\delta+\left( \frac{1}{8}z+w\right) \frac{2}{5}q_X^{\vee}.
		\end{equation*}
		Since $\iota^*A=A$, we obtain the system
		\begin{equation*}
			\left\{ \arraycolsep=1pt\def\arraystretch{1.2} \begin{array}{l}
				9x+\frac{3}{2}y+\frac{1}{2}z=x+\frac{y}{2},\\16x+2y+\frac{9}{8}z=\frac{z}{8},\\24x+\frac{7}{2}y+\frac{3}{2}z=\frac{y}{2}, \end{array}\right.
		\end{equation*}
		which gives the following condition:
		\begin{equation} \label{Eq4 t=2}
			16x+2y+z=0.
		\end{equation}
	\end{itemize}
	We look for $x, y, z, w \in \Z$ which satisfy $(\ref{Eq1 t=2})$, $(\ref{Eq2 t=2})$, $(\ref{Eq3 t=2})$ and $(\ref{Eq4 t=2})$. Since $-2 \le 8x+4y \le 10$ by $(\ref{Eq1 t=2})$ and $(\ref{Eq2 t=2})$, and $x, y \in \Z$, we have
	\begin{equation*}
		2x+y \in \{0, 1, 2\}.
	\end{equation*}
	\begin{itemize}
		\item Suppose that $2x+y=0$. By $(\ref{Eq4 t=2})$ we have $z=-12x$, and $(\ref{Eq3 t=2})$ becomes $0<w-x<\frac{3}{5}$. Since $w-x \in \Z$, this condition is never satisfied.
		\item Suppose that $2x+y=1$. By $(\ref{Eq4 t=2})$ we have $z=-12x-2$, and $(\ref{Eq3 t=2})$ becomes $-\frac{3}{10}<w-x<\frac{3}{10}$. Since $w-x \in \Z$, we get $x=w$, and $(\ref{Eq2 t=2})$ gives $x \in \{-1 , 0\}$. We obtain the following two solutions:
		\begin{equation*}
			\left\{ \arraycolsep=1pt\def\arraystretch{1.2} \begin{array}{l}
				x=0,\\y=1,\\z=-2,\\w=0, \end{array}\right.
			\qquad
			\text{and}
			\qquad
			\left\{\arraycolsep=1pt\def\arraystretch{1.2} \begin{array}{l} x=-1,\\y=3,\\z=10,\\w=-1.\end{array}\right
			.
		\end{equation*}
		which coincide with the effective classes $f^*\sigma_{1, 1}$ and $f^*\sigma_2$ respectively, see (\ref{eq pullback Schubert classes}) and \cite[Proposition 4.4.4]{novario2021ths}. Moreover, with the same technique it is possible to show that $f^*\sigma_{1, 1}$ and $f^*\sigma_2$ are reduced and irreducible.
		\item The case $2x+y=2$ is symmetric to the case $2x+y=0$, i.e., if $A$ is a class obtained in this case, then $A$ coincides with a class $B$ obtained in the case $2x+y=0$. Since there are no classes in the case $2x+y=0$, there are no classes in the cases $2x+y=2$. 
	\end{itemize}
	We conclude that if $D_1 \cap D_2$ is a reducible surface, then it is reduced with two irreducible components whose classes in $H^{2,2}(X, \Z)$ are $f^*\sigma_{1, 1}$ and $f^*\sigma_2$.
	
	Suppose now that $t \neq 2$ and $S_{2t}$ is generic. We want to show that the surface $D_1 \cap D_2$ is reduced and irreducible. By Lemma $\ref{lem technical}$ it is enough to show that we cannot have $[D_1 \cap D_2]=A+B$ for effective $A, B \in H^{2,2}(X, \Z)$ such that $\iota^*A=A$ and $\iota^*B=B$. The technique is the same seen before. By abuse of notation we write $D$ also for its class in $\Pic(X)$. Recall that $D=bh-a\delta$, where $(a, b)$ is the minimal solution of the Pell-type equation $P_t(-1)$. We have already remarked that the first value of $t$ different from $2$ which satisfies our assumptions is $t=10$, and in this case $a=3$, hence for other values of $t$ that we consider we have $a > 3$. Thus we can suppose that $t \ge 10$ and $a \ge 3$. Since\,\,$S_{2t}$ is generic by assumption, by Theorem \ref{thm basis H2,2(S2, Z)} we can write
	\begin{equation*}
		\begin{array}{l}
			A=\left( x+\frac{y}{2} \right) h^2 +\frac{z}{8}\delta^2-\frac{y}{2}h\delta+\left( \frac{1}{8}z+w \right) \frac{2}{5}q_X^{\vee}, \\[1ex]
			B=\left( b^2-x-\frac{y}{2} \right) h^2 +\left( a^2-\frac{z}{8} \right) \delta^2+ \left( -2ab+\frac{y}{2} \right) h\delta-\left( \frac{1}{8}z+w \right)  \frac{2}{5}q_X^{\vee},
		\end{array}
	\end{equation*}
	for some $x, y, z, w \in \Z$. Proceeding as for the case $t=2$ we obtain the following conditions:
	\begin{equation} \label{Eq 1 t neq 2}
		0 \le 6tx+3ty+z+10w \le 6tb^2-2a^2,
	\end{equation}
	\begin{equation} \label{Eq 2 t neq 2}
		0 \le 2tx+ty+z+10w \le 2,
	\end{equation}
	\begin{equation} \label{Eq 3 t neq 2}
		0 < (4t+8t^2b^2)x+(2t+4t^2b^2-4abt)y+(1+tb^2)z+20w<12,
	\end{equation}
	\begin{equation} \label{Eq4 t neq 2}
		8tdx+4(td-c)y+dz=0.
	\end{equation}
	where condition (\ref{Eq 1 t neq 2}) is given by $\langle A, h^2 \rangle \ge 0$ and $\langle B, h^2 \rangle \ge 0$, condition (\ref{Eq 2 t neq 2}) is given by $(\sigma + \bar{\sigma})^2 \rangle \ge 0$ and $\langle B, (\sigma + \bar{\sigma})^2 \rangle \ge 0$, condition (\ref{Eq 3 t neq 2}) is given by $\langle A, D^2 \rangle > 0$ and $\langle B, D^2 \rangle > 0$ and condition (\ref{Eq4 t neq 2}) is given by $\iota^*A=A$. Note that $(\ref{Eq 1 t neq 2})$ and $(\ref{Eq 2 t neq 2})$ implies $-\frac{1}{t} \le 2x+y \le 2b^2+\frac{1}{t}$. Since $2x+y \in \Z$ and $t \ge 10$ we have
	\begin{equation*}
		2x+y \in \{0, 1, \dots , 2b^2\}.
	\end{equation*}
	Note that, similarly to the case $t=2$ seen above, a class $A$ obtained by imposing $2x+y \in \{b^2+1, \dots , 2b^2\}$ coincide with a class $B$ obtained for $2x+y \in \{0, 1, \dots , b^2\}$. Hence it suffices to study $2x+y \in \{0, 1, \dots , b^2\}$.
	Suppose that $2x+y=k$, where $k \in \{0, 1, \dots, , b^2\}$. By $(\ref{Eq 2 t neq 2})$ we have $-tk \le z+10w \le 2-tk$, and since $z+10w \in \Z$ we have 
%	\begin{equation*}
%		-tk \le z+10w \le 2-tk,
%	\end{equation*} 
%	and since $z+10w \in \Z$ we have
	\begin{equation*}
		z+10w \in \{-tk, -tk+1, -tk+2\}.
	\end{equation*}
	Suppose that $z+10w=-tk$. Then $(\ref{Eq 3 t neq 2})$ gives, after some computations, $-4t^2b^2k < -4abty+a^2z < -4t^2b^2k+12$.
%	\begin{equation*}
%		-4t^2b^2k < -4abty+a^2z < -4t^2b^2k+12.
%	\end{equation*}
	Since $-4abty+a^2z \in \Z$, we have $-4abty+a^2z=-4t^2b^2k+h$, where $h \in \{1, 2, \dots, , 11\}$.
%	\begin{equation*}
%		-4abty+a^2z=-4t^2b^2k+h, \qquad h \in \{1, 2, \dots, , 11\}.
%	\end{equation*}
	With the help of a computer we obtain
	\begin{equation*}
		w=\frac{-5a^2tk+2ha^2-4tk+h}{10a^2}.
	\end{equation*}
	Then $w$ is an integer only if $4tk-h \equiv 0$ $(\text{mod}\,\,a^2)$.
%	\begin{equation*}
%		4tk-h \equiv 0 \qquad (\text{mod}\,\,a^2).
%	\end{equation*}
	If $k=0$, then $-h \equiv_{a^2} 0$ only if $h=9$ and $a=3$, being $a \ge 3$. This happens only when $t=10$, and we get $w=\frac{19}{10}$, which is not an integer. Suppose now that $k \neq 0$. Since $k \le b^2$, we have $4tk-h \le 4a^2+4-h$, hence in order to get $4tk-h \equiv_{a^2} 0$ we must have $4tk-h \in \{a^2, 2a^2, 3a^2, 4a^2\}$. 
%	\begin{equation*}
%		4tk-h \in \{a^2, 2a^2, 3a^2, 4a^2\}.
%	\end{equation*}	 
	If $4tk-h=a^2$, then $k=\frac{tb^2+h-1}{4t}$.
%	\begin{equation*}
%		k=\frac{tb^2+h-1}{4t}.
%	\end{equation*}
	If $h \neq 1, 11$, then $h-1$ is not divisible by $t \ge 10$, and $k$ is not an integer. If $h=1$, then $k=\frac{b^2}{4}$, which is not an integer by Proposition \ref{prop D=bh-adelta then b odd}. If $h=11$, then $h-1=10$ is divisible by $t \ge 10$ if and only if $t=10$, which implies $b=1$: thus $k=\frac{1}{2}$, which is not an integer.
	
	In a similar way it is possible to show that all the other remaining cases are not possible: we omit the details, which can be found in \cite[Appendix B]{novario2021ths}. We conclude that there are no effective classes $A, B \in H^{2,2}(X, \Z)$ such that $[D_1 \cap D_2]=A+B$, hence $D_1 \cap D_2$ is a reduced and irreducible surface.
\end{proof}
\section{Geometric description} \label{Section geometric description}
Let $X$ be the Hilbert square of a generic K3 surface of degree $2t$. Suppose that $X$ admits an ample divisor $D$ with $q_X(D)=2$ and $t \neq 2$. In this section, using Theorem \ref{thm class F and comm diagram} and Theorem \ref{thm irreducibility property}, we show that $\varphi_{|D|}: X \rightarrow Y \subset \Pn^5$ is a double cover of an EPW sextic $Y$, hence $X$ is a double EPW sextic. Our idea is to follow the strategy developed by O'Grady in \cite{o2008irreducible} with remarkable simplifications obtained thanks to the existence of the anti-symplectic involution given by Theorem \ref{thm BCNS Aut(X)}. We will omit proofs which are identical to the ones in \cite{o2008irreducible}.

Let $X$ and $D$ be as above. We introduce the following notation from \cite[$\S 4$]{o2008irreducible}. We choose an isomorphism $|D|^{\vee} \xrightarrow{\sim} \Pn^5$ and we denote by $f: X \dashrightarrow \Pn^5$ the composition $X \dashrightarrow|D|^{\vee} \xrightarrow{\sim} \Pn^5$: this is basically the rational map $\varphi_{|D|}$. Let $B$ be the base locus of $|D|$, and $\beta_B: \tilde{X} \rightarrow X$ be the blow-up of $X$ in $B$. We denote by $\tilde{f}: X \rightarrow \Pn^5$ the regular map which resolves the indeterminacies of $f$. Let $Y:=\text{Im}(\tilde{f})$, which is a closed subset of $\Pn^5$. We obtain a dominant map, which we call $f$ by abuse of notation:
\begin{equation*}
	f: X \dashrightarrow Y.
\end{equation*}
Note that $\Deg(f)$ is even by Theorem \ref{thm class F and comm diagram}. Let $Y_0$ be the interior of $\tilde{f}(X \setminus B)$, thus $Y_0 \subseteq Y$ is open and dense. Let $X_0:=(X \setminus B) \cap \tilde{f}^{-1}(Y_0)$, so $X_0 \subseteq X$ is open and dense. We call $f_0$ the restriction of $\tilde{f}$ to $X_0$, which is a regular surjective map:
\begin{equation*}
	f_0: X_0 \rightarrow Y_0.
\end{equation*}
We now summarize the main steps which will lead us to the main theorem of this section:
\begin{enumerate}[label=(\roman*)]
	\item $\Dim(Y)=4$.
	\item $\Deg(f) \cdot \Deg(Y) \le 12$ and the equality holds if and only if $\text{Bs}|D|=\emptyset$.
	\item Either $\Deg(f)=2$ and $\Deg(Y)=6$ or $\Deg(f)=4$ and $\Deg(Y)=3$. In particular $\text{Bs}|D|=\emptyset$ and $f$ is a morphism.
	\item The case $\Deg(f)=4$ and $\Deg(Y)=3$ never holds.
	\item The variety $X$ is a double EPW sextic.
\end{enumerate}
The following corollary of Theorem \ref{thm irreducibility property} will be fundamental to prove the other results of this section.
\begin{cor} \label{cor L cap Y0 reduced and irred}
	Let $X$ and $D$ be as in Theorem $\ref{thm irr property t neq 2}$, and keep notation as above.
	\begin{thmlist}
		\item If $L \subset \Pn^5$ is a linear subset of codimension at most $2$, then $L \cap Y_0$ is reduced and irreducible and, if non empty, it has pure codimension equal to $\textnormal{cod}(L, \Pn^5)$. \label{cor L cap Y0 reduced and irred 1}
		\item The base locus $B=\textnormal{Bs}|D|$ has dimension at most $1$. Let $B_{\textnormal{red}}$ be the reduced scheme associated to $B$ and $B^1_{\textnormal{red}}$ be the union of the irreducible components of $B_{\textnormal{red}}$ of dimension $1$. If $D_1, D_2, D_3 \in |D|$ are linearly independent, then $D_1 \cap D_2 \cap D_3$ has pure dimension $1$ and there exists a unique decomposition
		\begin{equation*} 
			[D_1 \cap D_2 \cap D_3]=\Gamma + \Sigma,
		\end{equation*}
		where $\Gamma$, $\Sigma$ are (classes of) effective $1$-cycles such that
		\begin{itemize}
			\item $\textnormal{Supp}(\Gamma) \cap B_{\textnormal{red}}$ is either $0$-dimensional or empty.
			\item $\textnormal{Supp}(\Sigma)=B^1_{\textnormal{red}}$.
		\end{itemize}
	\label{cor L cap Y0 reduced and irred 2}
	\end{thmlist}
\end{cor}
We only prove Item $(i)$, even if it is almost the same as \cite[Corollary 4.2.(i)]{o2008irreducible}, since in our setting it is a direct consequence of Proposition \ref{prop every divisor in |D| is prime} and of Theorem \ref{thm irr property t neq 2}.
\begin{proof}
	If $L \cong \Pn^5$, there is nothing to prove. If $\text{cod}(L, \Pn^5)=1$, let $D_1 \in |D|$ be the divisor which corresponds to $L$ in the isomorphism $|D|^{\vee} \cong \Pn^5$. Then $[D_1 \cap X_0]=f_0^*[L]$, where $[D_1 \cap X_0]$ and $[L]$ are the fundamental cohomological classes of $D_1 \cap X_0$ and $L$ respectively. Since $X_0$ is open and dense in $X$ and $f_0$ is surjective, Proposition \ref{prop every divisor in |D| is prime} implies that $L \cap Y_0$ is reduced and irreducible of pure codimension $1$ if non-empty. If $\text{cod}(L, \Pn^5)=2$, then $L=L_1 \cap L_2$ with $L_1, L_2 \subset \Pn^5$ hyperplanes, hence $[D_1 \cap D_2 \cap X_0] = f_0^*[L]$, where $D_1, D_2 \in |D|$ corresponds to $L_1$ and $L_2$ respectively. Hence $L \cap Y_0$ is reduced and irreducible of pure codimension $2$ if non-empty by Theorem \ref{thm irr property t neq 2}. For Item $(ii)$ see \cite[Corollary 4.2.(ii)]{o2008irreducible} (one needs Theorem \ref{thm irr property t neq 2} once again).
\end{proof}
Note that $[D_1 \cap D_2 \cap D_3]$ in the statement of Corollary\,\,$\ref{cor L cap Y0 reduced and irred 2}$ denotes the fundamental cycle of $D_1 \cap D_2 \cap D_3$, see \cite[$\S 1.5$]{fulton2013intersection}. The following proposition corresponds to \cite[Corollary 4.3]{o2008irreducible}.
\begin{prop}
	Let $X$ and $D$ be as in Theorem $\ref{thm irr property t neq 2}$. Keep notation as above. Then $\Dim(Y) \in \{3, 4\}$.
\end{prop}
We want to prove that the case $\Dim(Y)=3$ never holds. First of all, similarly to \cite[Proposition 4.5]{o2008irreducible}, we give boundaries to $\Deg(Y)$. We show the proof, which has little differences with the one by O'Grady: the anti-symplectic involution $\iota$ of Theorem \ref{thm BCNS Aut(X)} and Theorem \ref{thm class F and comm diagram} are strongly used.
\begin{prop} \label{prop dim Y=3, deg Y}
	Keep notation as above and suppose that $\Dim(Y)=3$. Then $3 \le \Deg(Y)\le 6$. Moreover, if $\Deg(Y)=6$, then the base locus $\textnormal{Bs}|D|$ is $0$-dimensional.
\end{prop}
\begin{proof}
	Let $d_Y:=\Deg(Y)$. Since $Y \subset \Pn^5$ is a non-degenerate subvariety, by \cite[Proposition 0]{eisenbud20163264} we have $d_Y \ge 3$. Let now $L_1, L_2, L_3 \subset \Pn^5$ be three generic hyperplanes. Then the intersection $Y \cap L_1 \cap L_2 \cap L_3$ is transverse and it is given by $d_Y$ points, which we call $p_1, \dots , p_{d_Y}$. Let $\Gamma_{0, i}:=f_0^{-1}(p_i)$ for $i=1, \dots , d_Y$, and let $\Gamma_i$ be the closure of $\Gamma_{0, i}$ in $X$. Let $D_1, D_2, D_3 \in |D|$ be the divisors which correspond to $L_1, L_2, L_3$ in the isomorphism $|D|^{\vee} \cong \Pn^5$. Since $L_1, L_2, L_3$ are generic, $D_1, D_2, D_3$ are linearly independent, and by Corollary $\ref{cor L cap Y0 reduced and irred 2}$ we have
	\begin{equation*} 
		[D_1 \cap D_2 \cap D_3] = \displaystyle\sum_{i=1}^{d_Y}\Gamma_i + \Sigma,
	\end{equation*}
	where the equality is in $H^{3,3}(X, \Z)$ and we write by abuse of notation $\Gamma_i$ and $\Sigma$ for their classes in $H^{3,3}(X, \Z)$. We write again $D$ also for its class in $\Pic(X)$: recall that $D=bh-a\delta$ with $(a, b)$ minimal solution of the Pell-type equation $P_t(-1)$, and $b$ is odd by Proposition \ref{prop D=bh-adelta then b odd}. Using Proposition \ref{Prop Hassett Tschinkel} one obtains $D^{\vee}=bh^{\vee}-2a\delta^{\vee}$ with $h^{\vee}=h$ and $\delta^{\vee}=\frac{\delta}{2}$ in $H^2(X, \Q)$. Moreover, as already remarked, the divisibility of $D$ in $H^2(X, \Z)$ is $\text{div}(D)=1$, hence by Proposition\,\,\ref{Prop Hassett Tschinkel} we have
	\begin{equation*}
		D \cdot D^{\vee} =\left( bh-a\delta, bh-2\frac{a}{2}\delta \right) =2.
	\end{equation*}  
	Since $D_1, D_2, D_3 \in |D|$ and $\iota^*D \cong D$, we have $\iota^*[D_1\cap D_2\cap D_3]=[D_1\cap D_2 \cap D_3]$, hence
	\begin{equation*}
		\displaystyle\sum_{i=1}^{d_Y}\iota^*\Gamma_i+\iota^*\Sigma=\displaystyle\sum_{i=1}^{d_Y}\Gamma_i+\Sigma.
	\end{equation*}
	By Theorem \ref{thm BCNS Aut(X)} and Theorem $\ref{thm H_2(S2, Z)}$ the only class fixed by the action induced by $\iota$ on $H^{3,3}(X, \Z)$ is $D^{\vee}$. Moreover, from Theorem $\ref{thm class F and comm diagram}$ we know that $f$ factors through the quotient by $\iota$, so $\iota^*\Gamma_i \cong \Gamma_i$, since $\Gamma_i=\overline{f_0^{-1}(p_i)}$. Then the class of $\Gamma_i$ is either some positive multiple of $D^{\vee}$, or it is of the form $\Gamma_i=\Gamma_i^{\prime}+\iota^*\Gamma_i^{\prime}$, where $\Gamma_i^{\prime}$ is an effective class. Since $q_X(D)=2$, which gives $\int_X c_1(\Ol_X(D))^4=12$ by Theorem \ref{thm BBF form}, we have
	\begin{equation*}
		12=\langle D, \displaystyle\sum_{i=1}^{d_Y}\Gamma_i+\Sigma \rangle.
	\end{equation*}
	Note that $\langle D, \Gamma_i \rangle$ is either $\langle D, \alpha D^{\vee} \rangle =2\alpha$, where $\alpha \in \Z_{\ge 1}$, or $\langle D, \Gamma_i^{\prime}+\iota^*\Gamma_i^{\prime} \rangle \ge 2$, being $D$ ample. Hence we have $12 \ge \langle D, \Sigma \rangle +2d_Y$, i.e., $d_Y \le 6$. Moreover, if $d_Y=6$, then $\Sigma=0$, otherwise $\langle D, \Sigma \rangle > 0$ by the Nakai--Moishezon criterion, thus the base locus $\textnormal{Bs}|D|$ is $0$-dimensional, since the $1$-dimensional component of $\textnormal{Bs}|D|$ is contained in $\Sigma$ by Corollary $\ref{cor L cap Y0 reduced and irred 2}$.
\end{proof}
Using exactly the same techniques in \cite{o2008irreducible}, with important simplifications due to Theorem \ref{thm class F and comm diagram}, we obtain the following proposition.
\begin{prop} \label{prop dim Y=4 etc}
	Let $X$ and $D$ be as in Theorem $\ref{thm irr property t neq 2}$. Consider the map $f: X \dashrightarrow Y \subset \Pn^5$ induced by the complete linear system $|D|$. Then $\Dim(Y)=4$. Moreover, $|D|$ is basepoint free, i.e., $f$ is a morphism, and one of the following holds.
	\begin{enumerate}[label=(\roman*)]
		\item $\Deg(f)=2$ and $\Deg(Y)=6$.
		\item $\Deg(f)=4$ and $\Deg(Y)=3$.
	\end{enumerate}
\end{prop}
\begin{proof}
	After having obtained Proposition \ref{prop dim Y=3, deg Y}, following \cite[$\S 5$]{o2008irreducible} one shows that $\Dim(Y)=4$. Then, one obtains results which correspond to \cite[Proposition 4.6, Corollary 4.7]{o2008irreducible}, and proceeding as in \cite[$\S 4$]{o2008irreducible} we have the following possibilities:
	\begin{enumerate}
		\item $\Deg(Y)=2$.
		\item $\Deg(Y)=3$, $\Deg(f)=3$, $\textnormal{Bs}|D| \neq \emptyset$.
		\item $\Deg(Y)=3$, $\Deg(f)=4$, $\textnormal{Bs}|D| = \emptyset$.
		\item $\Deg(Y)=4$, $\Deg(f)=3$, $\textnormal{Bs}|D| = \emptyset$.
		\item $\Deg(Y)=6$, $\Deg(f)=2$, $\textnormal{Bs}|D| = \emptyset$.
		\item $\Deg(f)=1$.
	\end{enumerate}
As already remarked, $\Deg(f)$ is even by Theorem \ref{thm class F and comm diagram}, hence (2), (4) and (6) are not possible. Case (1) does not hold, otherwise one can find a linear subset $L \subset \Pn^5$ of dimension $3$ such that $L \cap Y_0$ is reducible, which contradicts Corollary\,\,$\ref{cor L cap Y0 reduced and irred 1}$, see \cite[$\S 5.2$]{o2008irreducible} for details.
\end{proof}
Note that in \cite{o2008irreducible} much longer discussions are needed to show that (2) and (4) in the proof of Proposition\,\,\ref{prop dim Y=4 etc} never hold: the existence of the anti-symplectic involution and Theorem\,\,\ref{thm class F and comm diagram} simplifies the situation a lot. Moreover, recall that Corollary \ref{cor L cap Y0 reduced and irred 1} is a consequence of the irreducibility property of Theorem \ref{thm irr property t neq 2}. When $t=2$, the irreducibility property does not hold, see Theorem \ref{thm irr property t=2}, hence Corollary \ref{cor L cap Y0 reduced and irred 1} is not true and case (1) in the proof of Proposition\,\,\ref{prop dim Y=4 etc} is possible, actually it holds: as we have already seen in Section \ref{Section irreducibility property}, in this case $\Deg(f)=6$ and $\Deg(Y)=2$, in particular\,\,$Y$ is the Grassmannian $\mathbb{G}(1, \Pn^3)$ of lines in $\Pn^3$.

The final step is to show that Item $(ii)$ of Proposition\,\,\ref{prop dim Y=4 etc} is impossible and in Item $(i)$ the variety $Y \subset \Pn^5$ is an EPW sextic,\,\,$X$ is a double EPW sextic and $f$ is the double cover associated. See \cite{eisenbud2001lagrangian} for the definition of EPW sextic and \cite{o2006irreducible} for details on double EPW sextics. We remark that we cannot proceed following \cite[$\S 5.5$]{o2008irreducible}: in that case, the fact that the variety is a \emph{deformation} of the Hilbert square of a K3 surface plays a central role in the proof. If\,\,$M$ is an IHS manifold of $K3^{[2]}$-type and $h \in H^{1,1}(M, \Q)$ is an element such that $q_M(h) \neq 0$, O'Grady shows that $H^4(M, \C)$ can be decomposed as
\begin{equation*}
	H^4(M, \C)=\left( \C h^2 \oplus \C q_M^{\vee} \right) \oplus \left( \C h \otimes h^{\perp} \right) \oplus W(h),
\end{equation*}
where the orthogonality is with respect to the BBF form and $W(h):=(q_M^{\vee})^{\perp} \cap \text{Sym}^2(h^{\perp})$, where $(q_M^{\vee})^{\perp}$ is the orthogonal with respect to the intersection product. Then O'Grady takes an IHS manifold $X$ deformation equivalent to $M$ such that properties (1)-(6) of \cite[Proposition 3.2]{o2008irreducible} hold. In particular property (4) says that if $V \subset H^4(X)$ is a rational sub Hodge structure, then $V_{\C}=V_1 \oplus V_2 \oplus V_3$, where $V_1 \subset (\C h^2 \oplus \C q_X^{\vee})$, $V_2$ is either $0$ or equal to $\C h \otimes h^{\perp}$ and $V_3$ is either $0$ or equal to $W(h)$. In our case, there is no reason why a similar result holds, without deforming the variety $X=S^{[2]}_{2t}$ that we are studying: instead we exploit once again the anti-symplectic involution $\iota$ and Theorem\,\,\ref{thm class F and comm diagram}. We can finally prove the main theorem of this paper.
\begin{thm} \label{thm main}
	Let $X$ be the Hilbert square of a generic K3 surface $S_{2t}$ of degree $2t$ such that $X$ admits an ample divisor $D$ with $q_X(D)=2$. Suppose that $t \neq 2$, and denote by $\iota$ the anti-symplectic involution which generates $\Aut(X)$. Then the complete linear system $|D|$ is basepoint free, the morphism
	\begin{equation*}
		\varphi_{|D|}: X \rightarrow Y \subset \Pn^5
	\end{equation*}
	is a double cover whose ramification locus is the surface $F$ of points fixed by $\iota$, and $Y \cong X/\langle \iota \rangle$ is an EPW sextic, so\,\,$X$ is a double EPW sextic.
\end{thm}
\begin{proof}
	We show that $\Deg(f)=4$ and $\Deg(Y)=3$ in Proposition \ref{prop dim Y=4 etc} never holds. Suppose by contradiction that $\Deg(f)=4$ and $\Deg(Y)=3$. First we show that $Y \subset \Pn^5$ is a normal variety. Since $Y \subset \Pn^5$ is a hypersurface, by \cite[Proposition II.8.23]{hartshorne2013algebraic} we have that $Y$ is normal if and only if $\text{cod}_Y(\text{Sing}(Y)) \ge 2$. Suppose now that $\Dim(\text{Sing}(Y))=3$. Note that $Y$ does not contain planes, otherwise we get a contradiction with Corollary\,\,$\ref{cor L cap Y0 reduced and irred 1}$. As observed by O'Grady in \cite[Claim 5.10]{o2008irreducible}, since in $\Pn^5$ hyperplanes and quadrics contain planes, a cubic hypersurface of $\Pn^5$ which is either non-reduced or reducible contains planes. Hence in our case $Y$ is reduced and irreducible. Then, as in \cite[Lemma 5.17]{o2008irreducible}, the intersection between the variety $Y$ and a generic plane of $\Pn^5$ is a singular cubic curve which is reduced and irreducible. In particular this has only one singular point, so $\text{Sing}(Y)$ has exactly one irreducible component $\Sigma \cong \Pn^3$ of degree $1$. Thus $Y \supset \Sigma$, and it contains planes. This contradicts Corollary $\ref{cor L cap Y0 reduced and irred 1}$. We conclude that $\text{cod}_Y(\text{Sing}(Y)) \ge 2$, hence $Y$ is normal. The commutative diagram (\ref{eq comm diagram}) gives the following:
	\begin{equation*}
		\begin{tikzcd}
			X \arrow[rr, "f"]  \arrow[dr, "\pi^{\prime}"] & & Y \subset \Pn^5\, \\ 
			& X/\langle \iota \rangle. \arrow[ur, "\bar{f}"] &
		\end{tikzcd}
	\end{equation*}
	Since $\Deg(f)=4$ by assumption, by the commutativity of the diagram above we have $\Deg(\bar{f})=2$. Let $\tilde{Y}:=X/\langle \iota \rangle$. Since $X$ is smooth and $\Aut(X) \cong \langle \iota \rangle$ is a finite group, the quotient $\tilde{Y}$ is a normal variety. Both $\tilde{Y}$ and $Y$ are normal varieties, hence by \cite[Remark 2.4]{debarre2019double} there is a non trivial involution $\tau: \tilde{Y} \rightarrow \tilde{Y}$ such that $Y \cong \tilde{Y}/\langle \tau \rangle$. We show that $\tau$ lifts to an automorphism on $X$. We use a technique from \cite[p.179]{o2013double} and \cite[Proposition B.8]{debarre2018gushel}. Let $F=\textnormal{Fix}(\iota)$ be the fixed locus of the anti-symplectic involution $\iota$. Since $\text{Sing}(\tilde{Y})=\pi^{\prime}(F)$ and $\tau$ is an automorphism on $\tilde{Y}$, we have $\tau(\text{Sing}(\tilde{Y}))=\text{Sing}(\tilde{Y})$. Let $Y^{\prime}:=\tilde{Y}\setminus \text{Sing}(\tilde{Y})$. Then the restriction of $\tau$ to $Y^{\prime}$ gives an involution $\tau|_{Y^{\prime}}: Y^{\prime} \rightarrow Y^{\prime}$ of $Y^{\prime}$. We set $\tau^{\prime}:=\tau|_{Y^{\prime}}$. Since $\text{cod}_X(F)=2$ and $X$ is simply connected, we have that $X^{\prime}:=X\setminus F$ is simply connected. Thus if we restrict $\pi^{\prime}$ to $X^{\prime}$ we obtain the following universal cover:
	\begin{equation*}
		\pi^{\prime}|_{X^{\prime}}: X^{\prime} \rightarrow Y^{\prime}.
	\end{equation*}
	By the lifting criterion $\tau^{\prime}$ lifts to an automorphism on $X^{\prime}$. Thus we obtain a birational map $\tilde{\tau}: X \dashrightarrow X$ which is not defined a priori on the fixed locus $F$. By \cite[Proposition B.3]{debarre2019period} we have $\text{Bir}(X) \cong \Aut(X) \cong \langle \iota \rangle$, so $\tilde{\tau}$ is either the identity or $\iota$. This implies that the involution $\tau: \tilde{Y} \rightarrow \tilde{Y}$ is the identity, which is a contradiction. We conclude that $\Deg(\tilde{f})$ cannot be $2$, hence we get $\Deg(f)=2$ and $\Deg(Y)=6$. It remains to show that $Y$ is an EPW sextic and $f$ is a double cover ramified over $F$, so that $X$ is a double EPW sextic. This is true by \cite[Theorem 1.1]{o2006irreducible}, and we are done.
\end{proof}
\newcommand{\etalchar}[1]{$^{#1}$}

\end{document}